\documentclass[11pt]{amsart}
\usepackage{amsmath, amssymb, amscd, mathrsfs, url, pinlabel, verbatim}
\usepackage[margin=1.25in,marginparwidth=1in,centering,letterpaper,dvips]{geometry}
\usepackage{color,dcpic,latexsym,graphicx,epstopdf,comment}
\usepackage{stmaryrd}
\usepackage[all]{xy}
\usepackage{tikz,tikz-cd,pgfplots,calc}
\usepackage{enumitem,upquote,tabularx,textcomp}
\usepackage[pagebackref]{hyperref}
\usepackage[all]{hypcap}
\usepackage[color=blue!20!white,textsize=tiny]{todonotes}

\title{Framed instanton homology and concordance, II}

\author{John A. Baldwin}
\email{john.baldwin@bc.edu}
\address{Department of Mathematics\\Boston College}

\author{Steven Sivek}
\email{s.sivek@imperial.ac.uk}
\address{Department of Mathematics\\Imperial College London}

\thanks{JAB was supported by NSF CAREER Grant DMS-1454865 and NSF FRG Grant DMS-1952707.}

\makeatletter
\newtheorem*{rep@theorem}{\rep@title}
\newcommand{\newreptheorem}[2]{%
\newenvironment{rep#1}[1]{%
 \def\rep@title{#2 \ref{##1}}%
 \begin{rep@theorem}}%
 {\end{rep@theorem}}}
\makeatother

\newtheorem {theorem}{Theorem}
\newreptheorem{theorem}{Theorem}
\newtheorem {lemma}[theorem]{Lemma}
\newtheorem {proposition}[theorem]{Proposition}
\newtheorem {corollary}[theorem]{Corollary}

\numberwithin{equation}{section}
\numberwithin{theorem}{section}

\theoremstyle{definition}

\newtheorem{remark}[theorem]{Remark}
\newtheorem*{remark*}{Remark}

\newlist{pcases}{enumerate}{1}
\setlist[pcases]{
  label=\bf{Case~\arabic*:}\protect\thiscase.~,
  ref=\arabic*,
  align=left,
  labelsep=0pt,
  leftmargin=0pt,
  labelwidth=0pt,
  parsep=0pt
}
\newcommand{\case}[1][]{%
  \if\relax\detokenize{#1}\relax
    \def\thiscase{}%
  \else
    \def\thiscase{~#1}%
  \fi
  \item
}

\newcommand{\Z}{\mathbb{Z}}

\newcommand{\R}{\mathbb{R}}

\newcommand{\CP}{\mathbb{CP}}

\newcommand{\Q}{\mathbb{Q}}
\newcommand{\spc}{\operatorname{Spin}^c}

\newcommand{\lsl}{\mathit{sl}}

\newcommand{\maxsl}{\overline{\lsl}}
\newcommand{\rank}{\operatorname{rank}}

\newcommand{\cA}{\mathcal{A}}
\newcommand\cB{\mathcal{B}}
\newcommand\cC{\mathcal{C}}

\newcommand{\sD}{\mathscr{D}}

\newcommand{\cG}{\mathcal{G}}

\newcommand{\cptwo}{\overline{\CP^2}}

\newcommand\hfred{{\mathit{HF}_{\mathrm{red}}}}

\DeclareFontFamily{U}{mathx}{\hyphenchar\font45}
\DeclareFontShape{U}{mathx}{m}{n}{
      <5> <6> <7> <8> <9> <10>
      <10.95> <12> <14.4> <17.28> <20.74> <24.88>
      mathx10
      }{}
\DeclareSymbolFont{mathx}{U}{mathx}{m}{n}
\DeclareFontSubstitution{U}{mathx}{m}{n}
\DeclareMathAccent{\widecheck}{0}{mathx}{"71}

\newcommand{\img}{\operatorname{Im}}

\newcommand{\hfhat}{\widehat{\mathit{HF}}}

\newcommand{\gr}{\operatorname{gr}}

\newcommand{\id}{\operatorname{id}}

\newcommand{\pt}{\mathrm{pt}}

\newcommand{\godd}{\mathrm{odd}}

\newcommand{\mirror}[1]{\overline{#1}}
\newcommand{\cinvt}{\nu^\sharp}
\newcommand{\chominvt}{\tau^\sharp}
\newcommand{\cepsilon}{\epsilon^\sharp}

\newcommand{\Kh}{\mathrm{Kh}}
\newcommand{\Khr}{\overline{\Kh}}
\newcommand{\Khodd}{\Kh'}
\newcommand{\Khoddr}{\overline{\Khodd}}
\newcommand{\dcover}{\Sigma_2}

\newcommand{\bfone}{{\bf 1}}

\usetikzlibrary{calc,intersections}
\tikzset{every picture/.style=thick}
\tikzset{link/.style = { white, double = black, line width = 1.75pt, double distance = 1.25pt, looseness=1.75 }}
\tikzset{crossing/.style = {draw, circle, dotted, minimum size=0.5cm, inner sep=0, outer sep=0}}
\pgfplotsset{compat=1.12}

\begin{document}

\begin{abstract}
We continue our study of the integer-valued knot invariants $\cinvt(K)$ and $r_0(K)$, which together determine the dimensions of the framed instanton homologies of all nonzero Dehn surgeries on $K$.  We  first establish a   ``conjugation" symmetry for the decomposition of cobordism maps  constructed in our earlier work, and  use this to prove,  among many other things, that $\cinvt(K)$ is always either zero or odd. We then  apply these technical results to study linear independence  in the homology cobordism group, to define an instanton Floer analogue $\cepsilon(K)$ of Hom's $\epsilon$-invariant in Heegaard Floer homology,  and to the problem of characterizing a given 3-manifold as  Dehn surgery on a knot in $S^3$.
\end{abstract}

\maketitle

\section{Introduction} \label{sec:intro}

In \cite{bs-concordance}, we defined for any knot $K \subset S^3$ a smooth concordance invariant
\[ \cinvt(K) \in \Z \]
which defines a quasi-morphism on the smooth concordance group.  This integer, together with another invariant $r_0(K) \in \Z$ which satisfies \[r_0(K) \geq |\cinvt(K)|\textrm{ and }r_0(K) \equiv \cinvt(K) \pmod{2},\] governs the framed instanton homology \cite{km-yaft} of surgeries on $K$ as follows:

\begin{theorem}[{\cite[Theorem~1.1]{bs-concordance}}] \label{thm:dim-surgery}
For any knot $K \subset S^3$, we have
\[ \dim_\Q I^\#(S^3_{p/q}(K)) = q\cdot r_0(K) + |p-q\cinvt(K)| \]
where $p$ and $q$ are relatively prime integers with $q > 0$, and where if $\cinvt(K)=0$ then $p\neq 0$.
\end{theorem}

\noindent It may be difficult to compute $\cinvt(K)$ in general, but for example if $K$ is an instanton L-space knot then it follows from \cite{bs-lspace,lpcs} that $\cinvt(K)=2g(K)-1$; in particular, $\cinvt(K)$ is an odd integer in this case.

In this paper, we give an alternative formulation of $\cinvt(K)$, in terms of the decomposition of cobordism maps in framed instanton homology  which we developed in  earlier work \cite{bs-lspace}.  After proving a new ``conjugation'' symmetry for this decomposition (Theorem~\ref{thm:conjugation}), named after analogous symmetries in Heegaard Floer homology and  monopole Floer homology, we then show that the ``L-space oddity'' observed above generalizes to nearly all knots:

\begin{theorem}[Theorem~\ref{thm:nu-odd}] \label{thm:main-nu-odd}
For any knot $K \subset S^3$, $\cinvt(K)$ is either zero or odd.
\end{theorem}

We use these results to study the question of when the invariant $r_0(K)$ takes small values.  Note that if $r_0(K) \leq 2$ then by definition we must have
\[ (\cinvt(K), r_0(K)) = (0,0) \text{ or } (\pm1,1) \text{ or } (0,2). \]
The first two cases are addressed by \cite[Proposition~7.12]{bs-lspace}, and correspond to the unknot and the trefoils, respectively.  In Proposition~\ref{prop:almost-cinvt-0}, we understand the third case completely, resulting in the following theorem:

\begin{theorem} \label{thm:main-r0-small}
If $r_0(K) \leq 2$ then $K$ is the unknot, a trefoil, or the figure eight knot.
\end{theorem}

We apply the technical results above to various questions related to homology cobordism, Dehn surgery, and the framed instanton homology of surgeries, as described in the remainder of this introduction. In preparation for stating these applications,  recall that in \cite{bs-concordance} we also defined a smooth concordance \emph{homomorphism} $\chominvt(K)$ via the  homogenization of $\cinvt(K)$:
\[ \chominvt(K) = \frac{1}{2} \lim_{n\to\infty} \frac{\cinvt(\#^n K)}{n}. \]  Ghosh, Li, and Wong recently proved  \cite{glw}  that $\chominvt(K)$ is always an integer, a fact which we use in some of the applications discussed below.
\subsection{Homology cobordism}

Let $\Theta^3_\Z$ denote the  homology cobordism group, whose elements are integer homology 3-spheres modulo smooth integer homology cobordism.  Building on recent work of Nozaki, Sato, and Taniguchi \cite{nozaki-sato-taniguchi}, we prove the following:

\begin{theorem}[{Theorem~\ref{thm:homology-cobordism}}] \label{thm:homology-cobordism-main}
Let $K \subset S^3$ be a knot satisfying $\chominvt(K) > 0$, or more generally $\cinvt(K) > 0$. Then the homology spheres
\[ S^3_{1/n}(K), \quad n \geq 1 \]
are linearly independent in $\Theta^3_\Z$.
\end{theorem}

Most of the real content of Theorem~\ref{thm:homology-cobordism-main} is already in \cite{nozaki-sato-taniguchi}, where they prove the same conclusion under the hypothesis that the Fr{\o}yshov invariant \cite{froyshov} of $S^3_1(K)$ is negative.  Our contribution  is to show that their hypothesis is implied by $\chominvt(K) > 0$, or by $\cinvt(K) > 0$, either of which is often much easier to verify. For example, we have the following:

\begin{corollary} \label{cor:knots-with-tau-positive}
Let $K \subset S^3$ be any of the following:
\begin{itemize}
\item A knot having a transverse representative with positive self-linking number;
\item A  quasi-positive knot which is \emph{not} smoothly slice;
\item An alternating knot with negative signature, under the convention that the right-handed trefoil has signature $-2$.
\end{itemize}
Then the homology spheres $S^3_{1/n}(K)$ for $n \geq 1$ are linearly independent in $\Theta^3_\Z$.
\end{corollary}

This corollary generalizes several previous results in the literature. For instance, it applies to the  torus knots $T(p,q)$ with $p,q>1$, as these are quasi-positive and non-slice, recovering the old  result of Fintushel--Stern \cite{fs-seifert} and Furuta \cite{furuta} that  the positive $1/n$-surgeries on any such knot are linearly independent in $\Theta^3_\Z$. Moreover, Nozaki--Sato--Taniguchi used their recent result in \cite{nozaki-sato-taniguchi} to prove that if $K_k$ is the 2-bridge knot corresponding to the rational number $2/(4k-1)$, with $k$ a non-negative integer, then the positive $1/n$-surgeries on the mirror of $K_k$  are linearly independent in $\Theta^3_\Z$. We note that  such knots are alternating with negative signature, and therefore fall under the hypotheses of Corollary \ref{cor:knots-with-tau-positive}.

\begin{proof}[Proof of Corollary \ref{cor:knots-with-tau-positive}]
In the first case, we apply the bound \[\cinvt(K) \geq \maxsl(K) \geq 1\] of \cite[Theorem~6.1]{bs-concordance}.  For the second case, quasi-positive knots satisfy \[\chominvt(K) = g_s(K) > 0\] by \cite[Corollary~1.7]{bs-concordance}.  And for the third case, alternating knots satisfy \[\chominvt(K) = -\sigma(K)/2 > 0\] by \cite[Corollary~1.10]{bs-concordance}.
\end{proof}

\begin{remark}As another corollary of the work needed to prove Theorem~\ref{thm:homology-cobordism-main}, we also show in Proposition~\ref{prop:rationally-slice} that rationally slice knots $K$ satisfy $\cinvt(K) = \chominvt(K) = 0$.
\end{remark}

\subsection{Bounds on surgery slopes}

Given a rational homology sphere $Y$, it is natural to ask whether $Y$ can be realized as Dehn surgery on a knot  $K$ in $S^3$. If so, then one would  like to determine all such $K$ as well as the corresponding surgery slopes.
We use Theorems~\ref{thm:dim-surgery} and \ref{thm:main-r0-small} to prove the following upper bound on the denominators of such slopes. Note that this bound depends on $Y$ but not on $K$:

\begin{theorem} \label{thm:q-bound}
Suppose that the rational homology sphere $Y$ is $p/q$-surgery on a nontrivial knot $K \subset S^3$, where $\gcd(p,q)=1$ and $q>0$.  If $K$ is not a trefoil or the figure eight, then
\[ q \leq \frac{1}{3} \dim_\Q I^\#(Y), \]
and if equality holds then $\dim_\Q I^\#(Y) = 3$ and $p/q$ is either $\pm1$ or $\pm3$.
\end{theorem}

\begin{proof}
Theorem~\ref{thm:main-r0-small} says that $r_0(K) \geq 3$, so then
\[ \dim I^\#(Y) = q\cdot r_0(K) + |p-q\cinvt(K)| \geq 3q + |p-q\cinvt(K)| \geq 3q. \]
If equality holds then we must have $r_0(K)=3$ and $p/q = \cinvt(K)$, which satisfies \[|\cinvt(K)| \leq r_0(K) = 3\] and which is an odd integer since it has the same parity as $r_0(K)$.  Thus, $|p/q| \in \{1,3\}$ and $\dim I^\#(Y) = 3q = 3$.
\end{proof}

\begin{remark} \label{rem:r0-equals-3}
Theorem~\ref{thm:q-bound} could be improved by including the cases where $r_0(K)=3$.  These knots were also classified some time after the initial version of this paper appeared, as \cite[Theorem~3.13]{bs-characterizing-5_2} in Heegaard Floer homology and \cite[\S8]{li-ye-2} for the instanton Floer analogue (both of which use the fact that $T_{2,5}$ is the only genus-2 L-space knot \cite{frw-cinquefoil}).  In either setting, if $r_0(K)=3$, then $K$ must be either $T_{2,5}$ or $5_2$ up to mirroring.  Thus $r_0(K) \leq 3$ if and only if $K$ has crossing number at most $5$; otherwise, the proof of Theorem~\ref{thm:q-bound} gives us the stronger $q \leq \frac{1}{4} \dim_\Q I^\#(Y)$, with equality only if $r_0(K)=4$ (hence $\cinvt(K)=0$, by Theorem~\ref{thm:main-nu-odd} and the evenness of $r_0(K)-\cinvt(K)$) and $\frac{p}{q}=\cinvt(K)=0$.
\end{remark}

\begin{remark}
Gainullin \cite[Theorem~7]{gainullin} gave a similar bound using Heegaard Floer homology; namely, that if $Y$ is $p/q$-surgery on a nontrivial knot in $S^3$, then 
\[ q \leq |H_1(Y)| + \dim_{\Z/2\Z} \hfred(Y). \]
One can show straightforwardly that the right side is greater than or equal to 
\[ \frac{1}{2}(\dim_{\Z/2\Z}\hfhat(Y) + |H_1(Y)|). \]
Given that $\hfhat(Y)$ and $I^\#(Y)$ are expected to have the same dimension, at least over a field of characteristic zero, Theorem~\ref{thm:q-bound} should provide a strictly smaller bound on $q$. \end{remark}

Given a link $L \subset S^3$, Scaduto \cite{scaduto} constructed a spectral sequence
\begin{equation} \label{eq:scaduto-ss}
\Khoddr(L) \Longrightarrow I^\#(-\dcover(L))
\end{equation}
from the reduced odd Khovanov homology of $L$ to the branched double cover of $L$ with orientation reversed.  Since $\dim I^\#$ does not depend on orientation, we can combine this with Theorem~\ref{thm:q-bound} to get the following purely combinatorial bound on surgery slopes.

\begin{corollary} \label{cor:main-bdc}
Let $Y = \dcover(L)$ be a rational homology sphere which can be realized as $p/q$-surgery on a nontrivial knot $K \subset S^3$ with $\gcd(p,q)=1$ and $q>0$.  If $K$ is not a trefoil or the figure eight, then
\[ q \leq \frac{1}{3} \dim_\Q \Khoddr(L). \]
If equality holds then $p/q$ is either $\pm 1$ or $\pm 3$,  and $L$ is a knot with $\dim_\Q \Khoddr(L) = 3$.
\end{corollary}

The last claim that $L$ must be a knot follows because $\det(L) = |H_1(Y)| = |p|$ is odd.

\begin{remark}
Corollary~\ref{cor:main-bdc} is stronger than anything similar which might be proved with Heegaard Floer homology.  The reason for this is that the analogous spectral sequence $\Khr(L) \Longrightarrow \hfhat(-\dcover(L))$ \cite{osz-branched} is only proved with $\Z/2\Z$ coefficients, so it would yield a bound of the form
\[ q \leq \frac{1}{3} \dim_{\Z/2\Z} \Khr(L) = \frac{1}{3} \dim_{\Z/2\Z} \Khoddr(L). \]
This upper bound is larger than the one in Corollary~\ref{cor:main-bdc} whenever $\Khoddr(L)$ has $2$-torsion.
\end{remark}

\subsection{Framed instanton homology for different bundles}

One can define the framed instanton homology $I^\#(Y,\lambda)$ for any embedded multicurve $\lambda \subset Y$, where $\lambda$ represents the Poincar\'e dual of the first Chern class of a $U(2)$-bundle $E\to Y$.  This generalizes the usual $I^\#(Y)$, which refers to the case where $\lambda$ is nullhomologous.

The isomorphism class of $I^\#(Y,\lambda)$ depends only on that of the $SO(3)$-bundle $\mathrm{ad}(E)$, which is in turn determined by the homology class
\[ [\lambda] \in H_1(Y;\Z/2\Z). \]
However, as a corollary of Theorem~\ref{thm:main-nu-odd}, we can show that $I^\#(Y,\lambda)$ often does not even depend on this class, at least up to isomorphism of $\Z/2\Z$-graded vector spaces.

\begin{theorem}[{Theorem~\ref{thm:surgery-lambda}}] \label{thm:main-surgery-lambda}
Let $Y$ be a rational homology sphere which can be constructed by Dehn surgery on a knot in $S^3$.  Then
\[ \dim I^\#(Y,\lambda) \]
is independent of $\lambda$.
\end{theorem}

This can not be improved to a $\Z/4$-graded isomorphism, even in the case $Y=\mathbb{RP}^3$: Scaduto and Stoffregen \cite[\S5]{scaduto-stoffregen} showed that $I^\#(\mathbb{RP}^3,\lambda)$ is supported in grading $0 \pmod{4}$ when $[\lambda]=0$, but has rank 1 in gradings $0,2 \pmod{4}$ when $[\lambda]$ is nontrivial.

For zero-surgery, the dependence on the bundle is a  bit more subtle:

\begin{theorem}[{Corollary~\ref{cor:nu-0-twisted}}] \label{thm:main-zero-surgery}
Let $K \subset S^3$ be a knot with meridian $\mu$.  Then \[ \dim I^\#(S^3_0(K)) = \dim I^\#(S^3_0(K),\mu) \]
if and only if $\cinvt(K) \neq 0$.  If $\cinvt(K)=0$ then \[\dim I^\#(S^3_0(K))\neq \dim I^\#(S^3_0(K),\mu),\] in which case these dimensions are  $r_0(K)$ and $r_0(K)+2$ in some order.
\end{theorem}

\begin{remark} \label{rmk:splice}Theorem \ref{thm:main-zero-surgery} (or, more precisely, Theorem \ref{thm:nu-zero-surgery} from which it follows), was critical for our proof in \cite{bs-splice} that the splice of two nontrivial knots in  homology sphere instanton L-spaces is never a homology sphere instanton L-space.
\end{remark}

\subsection{An instanton Floer analogue of Hom's epsilon invariant}

Studying the framed instanton homology of zero-surgeries also allows us to better understand the knots $K$ with $\cinvt(K) = 0$, as in the following:

\begin{theorem}[Proposition~\ref{prop:add-nu-zero}] \label{thm:main-add-nu-zero}
If $K$ and $L$ are knots in $S^3$ such that $\cinvt(K) = 0$, then $\cinvt(K\#L) = \cinvt(L)$.  In particular, $\ker(\cinvt)$ is a subgroup of the smooth concordance group $\cC$.
\end{theorem}

Theorem~\ref{thm:main-add-nu-zero} then enables us to define an analogue
\[ \cepsilon(K) = 2\chominvt(K)-\cinvt(K) \]
of Hom's epsilon invariant $\epsilon$ in Heegaard Floer homology \cite{hom-epsilon,hom-ordering}.  We show in Proposition~\ref{prop:epsilon} that $\cepsilon$ shares many of the  properties that $\epsilon$ enjoys, and in particular leads to a total ordering on $\cC/\ker(\cinvt)$ which may be interesting.  We see it as more immediately useful for computing $\cinvt$ of connected sums: for example, Proposition~\ref{prop:epsilon} implies that $\cepsilon(K\#K)=\cepsilon(K)$ for all $K$, which quickly leads to the identity
\[ \chominvt(K) = \tfrac{1}{2}(\cinvt(K\#K)-\cinvt(K)). \]
We remark that this identity, and  Proposition~\ref{prop:epsilon} more generally, rely on the recent result of Ghosh, Li, and Wong \cite{glw} that $\chominvt(K)$ is always an integer.

\subsection*{Organization}

In Section~\ref{sec:conjugation}, we prove the promised conjugation symmetry, Theorem~\ref{thm:conjugation}, for the decomposition of cobordism maps in framed instanton homology.  In Section~\ref{sec:zero-surgery-maps}, we establish some notation and recall various results  about these cobordism maps. In Section~\ref{sec:nu-characterization}, we  reinterpret $\cinvt(K)$ in terms of the decomposition of the map associated to the trace of zero-surgery on $K$ (Proposition~\ref{prop:cinvt-count-y_i}), and refine the slice genus bound for $\cinvt(K)$ (Proposition~\ref{prop:max-y_j-bound}). We then use these results in Section~\ref{sec:parity} to prove our main result, Theorem~\ref{thm:main-nu-odd}, regarding the parity of $\cinvt(K)$. In Sections~\ref{sec:zero-surgery} and \ref{sec:bundles}, we apply these results to technical questions about the framed instanton homology of surgeries, proving Theorems   \ref{thm:main-surgery-lambda} and \ref{thm:main-zero-surgery}.  We use this apparatus in Section~\ref{sec:kernel-nu} to prove Theorem~\ref{thm:main-add-nu-zero} and to define our instanton analogue of Hom's epsilon invariant. The last two sections give broader topological applications of all of this technical work. In Section~\ref{sec:homology-cobordism}, we prove that either $\chominvt(K) > 0$ or $\cinvt(K) > 0$ implies $h(S^3_{-1}(K)) < 0$, and deduce Theorem~\ref{thm:homology-cobordism-main}. Finally, in Section~\ref{sec:almost-knots}, we study knots with $r_0(K)$  small, proving Theorem~\ref{thm:main-r0-small}; recall from the discussion above that this theorem then implies our bound on denominators of surgery slopes in Theorem \ref{thm:q-bound}.

\subsection*{Acknowledgments}

We thank Matt Hedden, Zhenkun Li, Tye Lidman,  Chris Scaduto, and Fan Ye for helpful conversations, and the referees for useful feedback on earlier versions of this paper. We particularly thank Zhenkun and Fan for telling us about their independent recent proof that if $\chominvt(K)$ is nonzero then $\cinvt(K)$ is odd \cite{li-ye}.

\section{Conjugation symmetry} \label{sec:conjugation}

Since this paper is a continuation of \cite{bs-concordance}, we will refer to \cite[\S2]{bs-concordance} for the necessary background on framed instanton homology.  See also \cite{km-yaft,scaduto} for more details.

The invariant $I^\#(Y,\lambda)$ is equipped with an absolute $\Z/2\Z$ grading \cite{froyshov,donaldson-book,scaduto} such that
\begin{equation} \label{eq:chi}
\chi(I^\#(Y,\lambda)) = \begin{cases} |H_1(Y;\Z)|, & b_1(Y)=0 \\ 0, & b_1(Y) > 0 \end{cases}
\end{equation}
\cite[Corollary~1.4]{scaduto}.  This can be lifted to a relative $\Z/4\Z$ grading on $I^\#(Y,\lambda)$, and if $[\lambda]=0$ in $H_1(Y;\Z/2\Z)$ then this lift can be made absolute \cite[\S7.3]{scaduto}.  If $[\lambda]$ is nonzero we can still make it absolute by choosing a Spin structure on $Y$, as described in \cite[\S2.2]{froyshov} or \cite[\S4]{scaduto-stoffregen}.  

There is also a natural \emph{eigenspace decomposition}
\[ I^\#(Y,\lambda) = \bigoplus_{s:H_2(Y)\to 2\Z} I^\#(Y,\lambda;s) \]
\cite{km-excision,bs-concordance}, in which each summand $I^\#(Y,\lambda;s)$ is the simultaneous generalized $s(h)$-eigenspace of a degree-$(-2)$ operator $\mu(h)$ as $h$ ranges over $H_2(Y;\Z)$.  Since each $\mu(h)$ has even degree, the $\Z/2\Z$ grading on $I^\#(Y,\lambda)$ descends to an absolute $\Z/2\Z$ grading on each eigenspace $I^\#(Y,\lambda;s)$.

With this setup in mind, we can define ``conjugation'' symmetries
\begin{equation} \label{eq:conjugation-y}
c_0,c_1,c_2,c_3: I^\#(Y,\lambda) \xrightarrow{\sim} I^\#(Y,\lambda)
\end{equation}
such that $c_i$ changes the signs in gradings $i+2$ and $i+3$ modulo 4, i.e.,
\begin{align*}
c_0(x_0,x_1,x_2,x_3) &= (x_0,x_1,-x_2,-x_3) \\
c_1(x_0,x_1,x_2,x_3) &= (-x_0,x_1,x_2,-x_3) \\
c_2(x_0,x_1,x_2,x_3) &= (-x_0,-x_1,x_2,x_3) \\
c_3(x_0,x_1,x_2,x_3) &= (x_0,-x_1,-x_2,x_3).
\end{align*}
(We note that $c_{i+2} = -c_i$ for all $i\in\Z/4\Z$.)  The fact that $\deg(\mu(h))=-2$ for all $h$ implies that each of these sends a generalized $\mu(h)$-eigenspace with eigenvalue $s(h)$ to one with eigenvalue $-s(h)$, giving an isomorphism
\begin{equation} \label{eq:conjugation-c_i}
c_i: I^\#(Y,\lambda;s) \xrightarrow{\sim} I^\#(Y,\lambda;-s)
\end{equation}
for all $i \in \Z/4\Z$ and all $s:H_2(Y;\Z)\to 2\Z$, with $c_i^2 = \id$, as in \cite[Theorem~2.25]{bs-lspace}.

Given a cobordism $(X,\nu): (Y_0,\lambda_0) \to (Y_1,\lambda_1)$ with absolute $\Z/4\Z$ gradings on each $I^\#(Y_i,\lambda_i)$, the map $I^\#(X,\nu)$ has a well-defined degree $d\in\Z/4\Z$, such that for any homogeneous elements
\begin{align*}
x &\in I^\#(Y_0,\lambda_0), &
y &\in I^\#(Y_1,\lambda_1),
\end{align*}
the matrix coefficient $\langle I^\#(X,\nu)(x), y\rangle$ is zero unless
\[ \gr(y) - \gr(x) \equiv d\pmod{4}. \]
An explicit formula for $d$ is given in \cite[\S7.3]{scaduto}, assuming that $[\lambda_j] = 0$ for each $j=0,1$, though we will not need it here.

Assuming that $b_1(X) = 0$, in \cite[Theorem~1.16]{bs-lspace} we defined a decomposition
\begin{equation*} 
I^\#(X,\nu) = \sum_{s: H_2(X)\to\Z} I^\#(X,\nu;s),
\end{equation*}
in which only finitely many summands are nonzero and each nonzero summand is a map
\[ I^\#(X,\nu;s): I^\#(Y_0,\lambda_0;s|_{Y_0}) \to I^\#(Y_1,\lambda_1;s|_{Y_1}) \]
between specific eigenspaces on either side.  Here $s|_{Y_i}$ denotes the composition
\[ H_2(Y_i) \to H_2(X) \xrightarrow{s} \Z, \]
where the first map is induced by inclusion.  The main result of this section is that these maps have the following ``conjugation symmetry'' property, comparable to the $\spc$ conjugation symmetry in other Floer homologies.

\begin{theorem} \label{thm:conjugation}
Let $(X,\nu): (Y_0,\lambda_0) \to (Y_1,\lambda_1)$ be a cobordism with $b_1(X)=0$, and suppose that we have fixed an absolute $\Z/4\Z$ grading on each $I^\#(Y_i,\lambda_i)$ which agrees with both the absolute $\Z/2\Z$ grading and relative $\Z/4\Z$ grading.  If $I^\#(X,\nu)$ has degree $d\in \Z/4\Z$, then the diagram
\[ \begin{tikzcd}
I^\#(Y_0,\lambda_0;s|_{Y_0}) \arrow[rr,"{I^\#(X,\nu;s)}"] \arrow[d,"\cong","c_k"'] &&
I^\#(Y_1,\lambda_1;s|_{Y_1}) \arrow[d,"c_{k+d}","\cong"'] \\
I^\#(Y_0,\lambda_0;-s|_{Y_0}) \arrow[rr,"{I^\#(X,\nu;-s)}"] &&
I^\#(Y_1,\lambda_1;-s|_{Y_1})
\end{tikzcd} \]
commutes for any $s: H_2(X;\Z) \to \Z$ and any $k\in\Z/4\Z$.
\end{theorem}

\begin{proof}
We recall the construction of the decomposition
\[ I^\#(X,\nu) = \sum_{s:H_2(X;\Z)\to\Z} I^\#(X,\nu;s) \]
from \cite{bs-lspace}.  In \cite[Theorem~5.1]{bs-lspace} we proved an analogue of Kronheimer and Mrowka's structure theorem \cite{km-structure}, asserting that the Donaldson series
\[ \sD_X^\nu(h) = D_{X,\nu}\left({-} \otimes \left(e^h + \frac{[\pt]}{2}e^h\right)\right) \]
is equal to a finite sum of exponentials of the form
\begin{equation} \label{eq:structure-theorem}
\sD_X^\nu(h) = e^{Q(h)/2} \sum_{j=1}^r a_{\nu,s_j}e^{s_j(h)},
\end{equation}
for some basic classes $s_j: H_2(X;\Z) \to \Z$ and maps
\[ a_{\nu,s_j}: I^\#(Y_0,\lambda_0) \to I^\#(Y_1,\lambda_1) \]
with rational coefficients.  We then defined
\[ I^\#(X,\nu;s) = \frac{1}{2}a_{\nu,s}, \]
which is understood to be zero if $s$ is not a basic class.

With this in mind, we now consider the series
\begin{align*}
\sD_X^\nu(-h) &= D_{X,\nu}\left({-} \otimes \left(e^{-h} + \frac{[\pt]}{2}e^{-h}\right)\right) \\
&= \sum_{n=0}^\infty D_{X,\nu}\left({-} \otimes \left(\frac{(-h)^n}{n!} + \frac{[\pt]}{2}\frac{(-h)^n}{n!}\right) \right) \\
&= \sum_{n\text{ even}} D_{X,\nu}\left({-} \otimes \left(\frac{h^n}{n!} + \frac{[\pt]}{2}\frac{h^n}{n!}\right) \right) -
\sum_{n\text{ odd}} D_{X,\nu}\left({-} \otimes \left(\frac{h^n}{n!} + \frac{[\pt]}{2}\frac{h^n}{n!}\right) \right).
\end{align*}
Since the classes $h\in H_2(X)$ and $[\pt] \in H_0(X)$ correspond to classes in $H^2(\cB)$ and $H^4(\cB)$ respectively, where $\cB = \cA/\cG$ is the configuration space of connections mod gauge on $X$, the terms with $n$ even and $n$ odd come from components of the ASD moduli space on $X$ with dimension $0\pmod{4}$ and $2\pmod{4}$, respectively.

Given a pair of flat connections $\alpha$ on $Y_0$ and $\beta$ on $Y_1$, each component of the ASD moduli space $\mathcal{M}(\alpha,\beta)$ which defines the corresponding matrix coefficient of $I^\#(X,\nu)$ has dimension $\gr(\beta)-\gr(\alpha) - d \pmod{4}$.  So for any pair of generators
\begin{align*}
x &\in I^\#(Y_0,\lambda_0), &
y &\in I^\#(Y_1,\lambda_1)
\end{align*}
which are each homogeneous with respect to the $\Z/4\Z$ grading, if we set $\delta = \gr(y)-\gr(x)$, then we have
\begin{equation} \label{eq:conjugation-cases}
\langle\sD_X^\nu(-h)\big(x\big), y\rangle = \begin{cases}
\langle\sD_X^\nu(h)\big(x\big), y\rangle, & \delta\equiv d \hphantom{+2\ }\pmod{4} \\
-\langle\sD_X^\nu(h)\big(x\big), y\rangle, & \delta\equiv d+2\pmod{4} \\
0, & \text{otherwise}.
\end{cases}
\end{equation}

Replacing $h$ with $-h$ in \eqref{eq:structure-theorem}, and noting that $Q(h)=Q(-h)$, now gives us
\[ \sD_X^\nu(-h) = e^{Q(h)/2} \sum_{j=1}^r a_{\nu,s_j}e^{-s_j(h)} \]
and hence
\[ \langle\sD_X^\nu(-h)(x), y\rangle = e^{Q(h)/2} \sum_{j=1}^r \langle a_{\nu,s_j}(x),y\rangle e^{-s_j(h)} \]
as a function of $h\in H_2(X;\Z)$.  By \eqref{eq:conjugation-cases} this is equal to some $\epsilon_{\delta-d} \in \{\pm 1\}$ times
\[ \langle\sD_X^\nu(h)(x), y\rangle = e^{Q(h)/2} \sum_{j=1}^r \langle a_{\nu,s_j}(x),y\rangle e^{s_j(h)}, \]
where
\[ \epsilon_n = \begin{cases} \hphantom{-}1, & n\equiv 0\pmod{4} \\ -1, & n\equiv 2\pmod{4}, \end{cases} \]
noting that both sides are identically zero if $\delta-d$ is odd.  Thus by comparing coefficients we have
\[ \langle a_{\nu,-s}(x),y \rangle = \epsilon_{\delta-d} \langle a_{\nu,s}(x),y \rangle. \]
Since this holds for any homogeneous $x$ and $y$, where $\epsilon_{\delta-d}$ is determined by $\delta = \gr(y)-\gr(x)$ as above, we can now check that
\[ a_{\nu,-s} \circ c_k = c_{k+d} \circ a_{\nu,s} \]
for any $k\in\Z/4\Z$ and so the theorem follows.
\end{proof}

\section{Cobordism maps associated to surgeries} \label{sec:zero-surgery-maps}

In this section we define some notation and recall some results which will be useful in understanding the invariant $\cinvt(K)$ in the following sections.

We begin by introducing Floer's exact triangle \cite{floer-surgery}, as developed for $I^\#$ by Scaduto \cite[\S7.5]{scaduto}.  Let $K$ be a framed knot in $Y$, with meridian $\mu$, and fix a multicurve $\lambda \subset Y\setminus K$.  Then there is a surgery exact triangle of the form
\begin{equation} \label{eq:surgery-triangle-general}
\dots \to I^\#(Y,\lambda) \xrightarrow{I^\#(X,\nu)} I^\#(Y_0(K),\lambda \cup \mu) \xrightarrow{I^\#(W,\omega)} I^\#(Y_1(K),\lambda) \to \dots.
\end{equation}
Here $X$ is the result of attaching a 2-handle to $[0,1] \times Y$ along $\{1\}\times K$, and the properly embedded surface $\nu \subset X$ is the union of $[0,1]\times\lambda$ and a meridional disk bounded by $\{1\}\times \mu$.  Similarly, $W$ is the trace of $(-1)$-framed surgery on a meridian $\mu_K$ of $K$, or more precisely on its image in $Y_0(K)$, and then $\omega$ is the union of $[0,1] \times (\lambda\cup\mu)$ and a disk bounded by a meridian $\mu'$ of $\{1\}\times\mu_K$.  (The target of $I^\#(W,\omega)$ is thus more accurately described as $I^\#(Y_1,\lambda\cup\mu\cup\mu')$, but $\mu\cup\mu'$ is zero in $H_1(Y_1;\Z/2\Z)$ and so this group is isomorphic to $I^\#(Y_1,\lambda)$.)

Now given a knot $K \subset S^3$, we define cobordisms
\[ X_n: S^3 \to S^3_n(K) \]
as the trace of $n$-surgery on $K$, for all $n \in \Z$.  Then the exact triangle \eqref{eq:surgery-triangle-general} specializes to
\begin{equation} \label{eq:surgery-triangle}
\dots \to I^\#(S^3) \xrightarrow{I^\#(X_n,\nu_n)} I^\#(S^3_n(K)) \xrightarrow{I^\#(W_{n+1},\omega_{n+1})} I^\#(S^3_{n+1}(K)) \to \dots,
\end{equation}
by taking $\lambda = \mu$ if $n$ is even and $\lambda = 0$ if $n$ is odd.  In the case $n=0$, we will also take $\lambda=0$ to get the triangle
\begin{equation} \label{eq:surgery-triangle-0}
\dots \to I^\#(S^3) \xrightarrow{I^\#(X_0,\tilde\nu_0)} I^\#(S^3_0(K),\mu) \xrightarrow{I^\#(W_1,\tilde\omega_1)} I^\#(S^3_1(K)) \to \dots,
\end{equation}
where $\mu$ is the image in $S^3_0(K)$ of a meridian of $K$ and $\tilde\nu_0$ is a properly embedded disk with boundary $\{1\}\times\mu$.

We will be especially interested in the decomposition of $I^\#(X_0,\tilde\nu_0)$, and of the maps $I^\#(X_n,\nu_n)$ for $n \geq 1$, into summands indexed by homomorphisms $H_2(X_n;\Z) \to \Z$, as in \cite{bs-lspace}.  For all $i \in \Z$, we thus define a homomorphism
\begin{equation} \label{eq:def-si}
s_i: H_2(X_0;\Z) \to \Z \qquad\text{by}\qquad s_i([\hat\Sigma]) = 2i,
\end{equation}
where $\hat\Sigma \subset X_0$ is the surface formed by gluing a Seifert surface for $\Sigma$ to the core of the $0$-framed 2-handle; and we let
\begin{align}
t_i = s_i|_{S^3_0(K)}: H_2(S^3_0(K);\Z) &\to 2\Z \label{eq:def-ti} \\
[\hat\Sigma] &\mapsto 2i. \nonumber
\end{align}
We can then define elements
\begin{align}
y_i = I^\#(X_0,\tilde{\nu}_0; s_i)\big(\bfone) &\in I^\#_\godd(S^3_0(K),\mu;t_i),  \label{eq:y_i} \\
y'_i = I^\#(X_0,\nu_0; s_i)\big(\bfone) &\in I^\#_\godd(S^3_0(K);t_i), \label{eq:tilde-y_i}
\end{align}
where $\bfone$ is a generator of $I^\#(S^3)$; if $K$ is nontrivial, then by \cite[Corollary~7.6]{km-excision}, these can only be nonzero when $|i| \leq g(K)-1$.

The following lemma will be our main application of Theorem~\ref{thm:conjugation}.

\begin{lemma} \label{lem:conjugation-y_i}
For any $i\in\Z$, we have $y_i \neq 0$ if and only if $y_{-i} \neq 0$.
\end{lemma}

\begin{proof}
Fixing an absolute $\Z/4\Z$ grading on $I^\#(S^3_0(K),\mu)$ and letting $d\in\Z/4\Z$ denote the resulting degree of the cobordism map $I^\#(X_0,\tilde\nu_0)$, Theorem~\ref{thm:conjugation} says that
\[ c_d(y_i) = c_d(I^\#(X_0,\tilde\nu_0;s_i)(\bfone)) = I^\#(X_0,\tilde\nu_0;s_{-i})(c_0(\bfone)) = y_{-i}, \]
since the generator $\bfone \in I^\#(S^3)$ has grading $0$ and is thus fixed by $c_0$.  Since $c_d$ is an isomorphism, the lemma follows.
\end{proof}

Following \cite[\S7]{bs-lspace}, we let
\[ \Sigma_n \subset X_n \]
be the generator of $H_2(X_n;\Z) \cong \Z$, built by gluing a Seifert surface for $K$ to a core of the $n$-framed 2-handle.  We observed in \cite[\S7.1]{bs-lspace} that $H_2(W_{n+1};\Z) \cong \Z$ is generated by a surface $F_{n+1}$, and that the diffeomorphism
\begin{equation} \label{eq:triangle-blowup}
X_n \cup_{S^3_n(K)} W_{n+1} \cong X_{n+1} \# \cptwo
\end{equation}
identifies
\begin{align} \label{eq:X-W-gens}
[\Sigma_n] &\longleftrightarrow [\Sigma_{n+1}] - [E], &
[F_{n+1}] &\longleftrightarrow [\Sigma_{n+1}] - (n+1)[E]
\end{align}
where $E \subset H_2(\cptwo;\Z)$ is the exceptional sphere.  We define homomorphisms
\begin{align}
s_{n,i}: H_2(X_n;\Z) &\to \Z &
u_{n+1,j}: H_2(W_{n+1};\Z) &\to \Z \label{eq:ts-homomorphisms} \\
[\Sigma_n] &\mapsto 2i-n &
[F_{n+1}] &\mapsto 2j, \nonumber
\end{align}
noting that the homomorphisms labeled $s_i$ and $t_i$ in \eqref{eq:def-si} and \eqref{eq:def-ti} coincide with $s_{0,i}$ and $s_{0,i}|_{S^3_0(K)}$, so that in this notation
\[ y_i = I^\#(X_0,\tilde\nu_0; s_{0,i})(\bfone); \]
and more generally, for $n \geq 1$ we define
\[ z_{n,i} = I^\#(X_n,\nu_n;s_{n,i})\big(\bfone\big) \in I^\#(S^3_n(K)). \]
For $n \geq 1$, the adjunction inequality says that
\begin{equation} \label{eq:z-adjunction-bounds}
z_{n,i} \neq 0 \ \Longrightarrow\ 1-g(K)+n \leq i \leq g(K)-1,
\end{equation}
as observed in \cite[Equation~(7.5)]{bs-lspace}.

\begin{lemma}[{\cite[Lemma~7.3]{bs-lspace}}] \label{lem:X-W-composition}
There are universal constants $\epsilon_n \in \{\pm 1\}$ such that
\begin{multline*}
I^\#(W_1,\tilde\omega_1;u_{1,j}) \circ I^\#(X_0,\tilde\nu_0;s_{0,i}) \\
= \frac{\epsilon_0}{2} \cdot (-1)^i \begin{cases}
I^\#(X_1,\nu_1;s_{1,i}) + I^\#(X_1,\nu_1;s_{1,i+1}), & j=i \\
0 & \text{otherwise}, \end{cases}
\end{multline*}
and
\begin{multline*}
I^\#(W_{n+1},\omega_{n+1};u_{n+1,j}) \circ I^\#(X_n,\nu_n;s_{n,i}) \\
= \frac{\epsilon_n}{2} \cdot \begin{cases}
\hphantom{-}I^\#(X_{n+1},\nu_{n+1};s_{n+1,i}), & j=i \\
-I^\#(X_{n+1},\nu_{n+1};s_{n+1,i+1}), & j=i-n \\
\hphantom{-}0 & \text{otherwise}
\end{cases}
\end{multline*}
for all $n\geq 1$ and all $i,j$.
\end{lemma}

The proof of Lemma~\ref{lem:X-W-composition} follows from careful application of \eqref{eq:X-W-gens} and the blow-up formula for cobordism maps.  Applying it to the generator $\bfone \in I^\#(S^3)$ and summing over all $j$ gives the following.

\begin{lemma} \label{lem:image-zni}
There are universal constants $\epsilon_n \in \{\pm 1\}$ such that
\[ I^\#(W_1,\tilde\omega_1)(y_i) = \frac{\epsilon_0}{2}\cdot (-1)^i(z_{1,i} + z_{1,i+1}) \]
for all $i$, and
\[ I^\#(W_{n+1},\omega_{n+1})(z_{n,i}) = \frac{\epsilon_n}{2}(z_{n+1,i}-z_{n+1,i+1}) \]
for all $n \geq 1$ and all $i$.
\end{lemma}

\section{A characterization of the $\cinvt$ invariant} \label{sec:nu-characterization}

In \cite{bs-concordance}, we defined the invariant $\cinvt(K)$ in terms of a simpler invariant,
\[ \cinvt(K) = N(K) - N(\mirror{K}), \]
where $N(K)$ is a nonnegative integer.  We defined $N(K)$ to be the least nonnegative integer $n$ such that $I^\#(X_n,\nu_n) = 0$, when taken with coefficients in a field of characteristic zero.  In this section we will give a new characterization of $N(K)$, and thus (in Proposition~\ref{prop:cinvt-count-y_i}) a useful description of $\cinvt(K)$ whenever $\cinvt(K)$ is positive.

Here and throughout this section we use the same notation as in Section~\ref{sec:zero-surgery-maps}.  For convenience we will use
\begin{align}
F_n &= I^\#(X_n,\nu_n) &
\tilde{F}_0 &= I^\#(X_0,\tilde\nu_0) \label{eq:def-F} \\
G_{n+1} &= I^\#(W_{n+1},\omega_{n+1}) &
\tilde{G}_1 &= I^\#(W_1,\tilde\omega_1)  \label{eq:def-G}
\end{align}
to denote the maps in the exact triangles \eqref{eq:surgery-triangle} and \eqref{eq:surgery-triangle-0} respectively.

\begin{proposition}[{\cite[Proposition~3.3]{bs-concordance}}] \label{prop:N-mirror}
Either $N(K) = N(\mirror{K}) = 1$, or at least one of $N(K)$ and $N(\mirror{K})$ is zero.  Moreover, if $\cinvt(K) \neq 0$ (so that $N(K) \neq N(\mirror{K})$) then
\[ \dim I^\#(S^3_0(K)) = \dim I^\#(S^3_0(K),\mu). \]
\end{proposition}

\begin{proposition}[{\cite[Lemma~7.5]{bs-lspace}}] \label{prop:ker-G}
Let $G_{n+1}$ and $\tilde{G}_1$ be the maps indicated in \eqref{eq:def-G}, which have even degree with respect to the $\Z/2\Z$ grading on $I^\#$.  For all $n > 0$, the kernel of the map
\[ G_n \circ G_{n-1} \circ \dots \circ G_2 \circ \tilde{G}_1: I^\#(S^3_0(K),\mu) \to I^\#(S^3_n(K)) \]
is a subspace of the span of the elements $y_i$, and likewise the kernel of
\[ G_n \circ G_{n-1} \circ \dots \circ G_2 \circ G_1: I^\#(S^3_0(K)) \to I^\#(S^3_n(K)) \]
lies in the span of the elements $y'_i$.  In both cases, equality holds for all $n \geq 2g(K)-1$.
\end{proposition}

Proposition~\ref{prop:ker-G} is only stated explicitly for the elements $y_i$ in \cite{bs-lspace}, but the proof in either case is the same after deciding whether to use the triangle \eqref{eq:surgery-triangle-0} (for the $y_i$) or the $n=0$ case of \eqref{eq:surgery-triangle} (for the $y'_i$).

\begin{proposition} \label{prop:cinvt-count-y_i}
If $\cinvt(K)$ is positive, then it is equal to
\[ \#\{i \mid y_i \neq 0\} = \#\{i \mid y'_i \neq 0 \}, \]
where the $y_i$ and $y'_i$ are the elements defined in \eqref{eq:y_i} and \eqref{eq:tilde-y_i} respectively.
\end{proposition}

\begin{proof}
Since $\cinvt(K)$ is positive, Proposition~\ref{prop:N-mirror} says that $N(\mirror{K}) = 0$ and that
\[ \cinvt(K) = N(K) \geq 1. \]
Let $N = N(K) \geq 1$.  Theorem~\ref{thm:dim-surgery} and Proposition~\ref{prop:N-mirror} tell us that
\[ \dim I^\#(S^3_N(K)) = \dim I^\#(S^3_0(K),\mu) - N = \dim I^\#(S^3_0(K)) - N, \]
so that the kernel of
\begin{equation} \label{eq:compose-G}
G_N \circ G_{N-1} \circ \dots \circ G_2 \circ \tilde{G}_1: I^\#(S^3_0(K),\mu) \to I^\#(S^3_N(K))
\end{equation}
has dimension at least $N$.  On the other hand, we have
\begin{align*}
\dim \ker(\tilde{G}_1) &= \dim \img(\tilde{F}_0) \leq 1, \\
\dim \ker(G_i) &= \dim \img(F_{i-1}) \leq 1 \text{ for all }i \geq 2
\end{align*}
by the exactness of \eqref{eq:surgery-triangle} and \eqref{eq:surgery-triangle-0}, so the kernel of \eqref{eq:compose-G} has dimension at most $N$ as well, and hence this dimension must be exactly $N$.  Likewise, we have
\[ \dim \ker\big( G_N \circ G_{N-1} \circ \dots \circ G_2 \circ G_1: I^\#(S^3_0(K)) \to I^\#(S^3_N(K))\big) = N \]
by an identical argument.

Now all of the maps $F_j$ in \eqref{eq:def-F} are zero for $j \geq N$ by definition, and hence the maps
\[ G_{N+1}, G_{N+2}, G_{N+3}, \dots \]
are all injective.  It follows for all $j > 2g(K)-1 \geq N$ that
\[ \dim \ker(G_j \circ G_{j-1} \circ \dots \circ G_2 \circ \tilde{G}_1) = \dim \ker(G_N \circ G_{N-1} \circ \dots \circ G_2 \circ \tilde{G}_1) = N \]
and likewise
\[ \dim \ker(G_j \circ G_{j-1} \circ \dots \circ G_2 \circ G_1) = \dim \ker(G_N \circ G_{N-1} \circ \dots \circ G_2 \circ G_1) = N. \]
According to Proposition~\ref{prop:ker-G}, the kernels on the left are precisely the spans of the elements $y_i$ and $y'_i$ respectively, and the nonzero $y_i$ and $y'_i$ are linearly independent because they lie in different eigenspaces
\[ I^\#(S^3_0(K),\mu; t_i) \subset I^\#(S^3_0(K),\mu) \quad\text{and}\quad I^\#(S^3_0(K); t_i) \subset I^\#(S^3_0(K)). \]
We conclude that there must be exactly $N = N(K) = \cinvt(K)$ nonzero elements $y_i$ and $N = N(K) = \cinvt(K)$ nonzero elements $y'_i$.
\end{proof}

\begin{remark} \label{rmk:W-shaped}
Suppose that $\cinvt(K) = 0$ but $\tilde{F}_0 = I^\#(X_0,\tilde\nu_0)$ is injective.  Then $\tilde{G}_1$ is surjective with 1-dimensional kernel, by the exactness of \eqref{eq:surgery-triangle-0}; and $G_j$ is injective for all $j\geq 2$, because the definition of $\cinvt$ and Proposition~\ref{prop:N-mirror} imply that $N(K) \leq 1$ and hence $F_{j-1} = 0$ for all $j \geq 2$.  Moreover, some of the $y_i$ must be nonzero, since their sum is $\tilde{F}_0(\bfone) \neq 0$.  Repeating the argument of Proposition~\ref{prop:cinvt-count-y_i}, we must have
\[ \dim \ker (G_n \circ \dots \circ G_2 \circ \tilde{G}_1) = \dim \ker(\tilde{G}_1) = 1 \]
for all $n \geq 2$.  Proposition~\ref{prop:ker-G} then says that there is exactly one nonzero $y_i$.
\end{remark}

As a first application of our new description of $\cinvt(K)$, we refine the slice genus bound $\cinvt(K) \leq \max(2g_s(K)-1,0)$ proved in \cite{bs-concordance}.

\begin{proposition} \label{prop:max-y_j-bound}
Suppose that $\cinvt(K) \geq 1$, and define
\[ m = \max \{ j \mid y_j \neq 0 \}. \]
Then $\cinvt(K) \leq 2m+1 \leq 2g_s(K)-1$, where $g_s(K)$ is the smooth slice genus of $K$.
\end{proposition}

\begin{proof}
Both $y_m$ and $y_{-m}$ are nonzero by Lemma~\ref{lem:conjugation-y_i}, and $g_s(K) \geq 1$ because if $K$ were slice then we would have had $\cinvt(K)=0$.  Thus if $\cinvt(K)=1$ then by Proposition~\ref{prop:cinvt-count-y_i} we must have $m=0$, and the desired inequalities become $1 \leq 1 \leq 2g_s(K)-1$, which is certainly true.  From now on we can assume that $\cinvt(K) > 1$.

Since $\cinvt(K) > 1$, we know from Proposition~\ref{prop:cinvt-count-y_i} that some $y_j$ other than $y_0$ is nonzero, so now Lemma~\ref{lem:conjugation-y_i} guarantees that $m \geq 1$.  It also implies that any given $y_j$ can only be nonzero if $-m \leq j \leq m$, so we invoke Proposition~\ref{prop:cinvt-count-y_i} again to show that $\cinvt(K) \leq 2m+1$.

We now repeat the argument of \cite[Proposition~7.7]{bs-lspace} with some minor modifications: both $y_m$ and $y_{-m}$ are nonzero, and the nonzero $y_i$ are linearly independent while their sum is the generator $I^\#(X_0,\tilde\nu_0)({\bf 1})$ of
\[ \ker\big(\tilde{G}_1: I^\#(S^3_0(K),\mu) \to I^\#(S^3_1(K))\big), \]
by the exactness of \eqref{eq:surgery-triangle-0}.  Thus neither $y_m$ nor $y_{-m}$ is in $\ker(\tilde{G}_1)$.  At the same time, since $y_j = 0$ for all $j > m$, Lemma~\ref{lem:image-zni} says that
\[ z_{1,j} + z_{1,j+1} = \pm 2\cdot \tilde{G}_1(y_j) = 0 \text{ for all } j > m, \]
and \eqref{eq:z-adjunction-bounds} says that $z_{1,j} = 0$ for all large enough $j$, so in fact we have $z_{1,j} = 0$ for all $j > m$.  But then
\[ z_{1,m} = z_{1,m} + z_{1,m+1} = \pm 2\cdot \tilde{G}_1(y_m) \neq 0, \]
so the map $I^\#(X_1,\nu_1;s_{1,m})$ which sends $\bfone \in I^\#(S^3)$ to $z_{1,m}$ is also nonzero.

Now $H_2(X_1)$ is generated by a surface $F$ of self-intersection 1 and genus $g_s(K) > 0$, built by capping off a minimal-genus slicing surface with the core of the $2$-handle, and by definition we have
\[ s_{1,m}([F]) = 2m-1 > 0. \]
Thus the adjunction inequality applied to $F$ and $I^\#(X_1,\nu_1;s_{1,m})$ says that
\[ (2m-1) + F\cdot F \leq 2g(F) - 2, \]
which is equivalent to the desired $2m+1 \leq 2g_s(K)-1$.
\end{proof}

\section{The parity of $\cinvt$} \label{sec:parity}

Our goal in this section is to prove the following theorem, stated in the introduction as Theorem \ref{thm:main-nu-odd}, using the characterization of $\cinvt(K)$ provided by Proposition~\ref{prop:cinvt-count-y_i}.

\begin{theorem} \label{thm:nu-odd}
For any knot $K \subset S^3$, the invariant $\cinvt(K) \in \Z$ is either zero or odd.
\end{theorem}

Lemma~\ref{lem:conjugation-y_i}, which asserts that $y_i$ is nonzero if and only if $y_{-i}$ is nonzero, allows us to reduce Theorem~\ref{thm:nu-odd} to the following proposition.  (Here and throughout this section we continue to use the notation of Section~\ref{sec:zero-surgery-maps}.)

\begin{proposition} \label{prop:y0-nonzero}
Let $K \subset S^3$ be a knot with $\cinvt(K) > 0$, or with $\cinvt(K) = 0$ and $I^\#(X_0,\tilde\nu_0)$ injective.  Then $y_0 \neq 0$.
\end{proposition}

We explain here how Proposition~\ref{prop:y0-nonzero} implies Theorem~\ref{thm:nu-odd} and give a quick application of Theorem~\ref{thm:nu-odd} to the invariant $\chominvt$, and then we will devote the remainder of this section to proving Proposition~\ref{prop:y0-nonzero}.

\begin{proof}[Proof of Theorem~\ref{thm:nu-odd}]
Supposing that $\cinvt(K)$ is nonzero, we may assume without loss of generality that $\cinvt(K) > 0$, because otherwise we use the identity $\cinvt(\mirror{K}) = -\cinvt(K)$ and replace $K$ with its mirror.  Proposition~\ref{prop:cinvt-count-y_i} then says that $\cinvt(K)$ is equal to the number of nonzero elements $y_i$.  But by Lemma~\ref{lem:conjugation-y_i}, the nonzero $y_i$ come in pairs $\{y_i,y_{-i}\}$, with the possible exception of $y_0$.  Since Proposition~\ref{prop:y0-nonzero} guarantees that $y_0$ is indeed nonzero, we conclude that $\cinvt(K)$ is odd.
\end{proof}

The following easy application of Theorem~\ref{thm:nu-odd} will be useful later.  In order to state it, we recall from \cite{bs-concordance} and the introduction that $\cinvt$ can be turned into a smooth concordance homomorphism $\cC \to \R$ by taking its homogenization
\[ \chominvt(K) = \frac{1}{2} \lim_{n\to\infty} \frac{\cinvt(\#^nK)}{n}. \]

\begin{lemma} \label{lem:tau-nonzero}
If $\chominvt(K) > 0$, then $\cinvt(K) > 0$.  Similarly, if $\chominvt(K) < 0$, then $\cinvt(K) < 0$.
\end{lemma}

\begin{proof}
Taking $\chominvt(K) > 0$ without loss of generality, by the definition of $\chominvt$ we must have
\[ \lim_{n\to\infty} \cinvt(\#^n K) = \infty. \]
Then there must be some integer $m$ such that
\[ \cinvt(\#^{m+1} K) > \cinvt(\#^m K) \geq 1, \]
and by Theorem~\ref{thm:main-nu-odd} both $\cinvt(\#^{m+1} K)$ and $\cinvt(\#^m K)$ are odd integers, so that
\begin{equation} \label{eq:cinvt-mp1-m}
\cinvt(\#^{m+1} K) - \cinvt(\#^m K) \geq 2.
\end{equation}
But $\cinvt(K)$ is a quasi-morphism \cite[Theorem~5.1]{bs-concordance} which satisfies the inequality
\begin{equation} \label{eq:quasi-morphism}
\left|\cinvt(L_1\#L_2) - \cinvt(L_1) - \cinvt(L_2) \right| \leq 1,
\end{equation}
so letting $L_1=\#^m K$ and $L_2=K$, we have
\begin{equation} \label{eq:cinvt-mp1-m-2}
\cinvt(\#^{m+1}K) - \cinvt(\#^mK) \leq \cinvt(K) + 1.
\end{equation}
We combine \eqref{eq:cinvt-mp1-m} and \eqref{eq:cinvt-mp1-m-2} to get $2 \leq \cinvt(K) + 1$, or $\cinvt(K) \geq 1$.
\end{proof}

\begin{remark}
Lemma~\ref{lem:tau-nonzero} also follows from the integrality of $\chominvt(K)$, proved by Ghosh, Li, and Wong \cite[Corollary~1.3]{glw}, together with the inequality $|2\chominvt(K)-\cinvt(K)| \leq 1$ of \cite[Proposition~5.4]{bs-concordance}.  Our proof requires fewer prerequisites, however, as it does not need to pass through any versions of instanton knot homology.
\end{remark}

\noindent With this application out of the way, we now turn to the proof of Proposition~\ref{prop:y0-nonzero}.

\subsection{Reduction to linear algebra} \label{ssec:reduction}

We wish to show that if $\cinvt(K)$ is positive, then the element
\[ y_0 \in I^\#(S^3_0(K),\mu;t_0) \]
must be nonzero.  The first step, carried out in this subsection, is to find a linear combination $x_n$ of the various $y_i$ whose image under $I^\#(W_n,\omega_n)\circ\dots\circ I^\#(W_1,\tilde{\omega}_1)$ is a nonzero multiple of the generator
\[ I^\#(X_n,\nu_n)(\bfone) \in I^\#(S^3_n(K)) \]
of $\ker I^\#(W_{n+1},\omega_{n+1})$.  This elaborates on a key step in our proof that instanton L-space knots are fibered \cite[Lemma~7.5]{bs-lspace}, where we showed that such linear combinations exist but did not find them explicitly.

Having found the various $x_n$, we will suppose that $y_0=0$ and package the $y_i$-coefficients of the elements $x_0,x_1,\dots,x_{\cinvt(K)-1}$ into a square matrix.  We will argue that this matrix must be invertible, because the $x_n$ will have to be linearly independent if their respective images in the various $I^\#(S^3_n(K))$ are to all be nonzero.  Then in Subsection~\ref{ssec:linalg} we will prove by a long, somewhat tedious computation that this matrix is not invertible after all, giving the desired contradiction.

In what follows, we will continue to write
\begin{align*}
\tilde{G}_1 &= I^\#(W_1,\tilde\omega_1), \\
G_n &= I^\#(W_n,\omega_n) \quad (n\geq 2)
\end{align*}
for simplicity, just as in \S\ref{sec:nu-characterization}.

\begin{lemma} \label{lem:X1-in-terms-of-z1}
Fix a knot $K \subset S^3$ with $\cinvt(K) > 0$ and an integer $h \geq g(K)-1$.  For $1 \leq n \leq 2h+1$, we define
\[ v_n = \sum_{i=-h}^h \binom{h-i}{n-1} z_{1,i} = \sum_{i=1-h}^{h-(n-1)} \binom{h-i}{n-1} z_{1,i} \]
as an element of $I^\#(S^3_1(K))$.  Then
\[ \sigma_n I^\#(X_n,\nu_n)(\bfone) = G_n\circ \dots \circ G_2 (v_n) \]
where $\sigma_n = \pm \frac{1}{2^{n-1}}$.
\end{lemma}

\begin{proof}
We define $\sigma_1=1$ and $\sigma_n = \frac{(-1)^{n-1}\epsilon_1\epsilon_2\dots\epsilon_{n-1}}{2^{n-1}}$ for $n \geq 2$, so that $\sigma_n = -\frac{\epsilon_{n-1}}{2}\sigma_{n-1}$ for all $n \geq 2$.  We claim by induction that for all $n \geq 1$ we have
\begin{equation} \label{eq:zni-recurrence}
\sigma_n \begin{bmatrix} z_{n,1-h} \\ z_{n,2-h} \\ z_{n,3-h} \\ \vdots \\ z_{n,h} \end{bmatrix} =
\begin{bmatrix} 0 & 0 & 0 & \dots & 0 \\ 1 & 0 & 0 & \dots & 0 \\ 1 & 1 & 0 & \dots & 0 \\ \vdots & \vdots & \vdots & \ddots & \vdots \\ 1 & 1 & 1 & \dots & 0 \end{bmatrix}^{n-1}
\begin{bmatrix} G_n\circ G_{n-1} \circ \cdots \circ G_2(z_{1,1-h}) \\ G_n\circ G_{n-1} \circ \cdots \circ G_2(z_{1,2-h}) \\ G_n\circ G_{n-1} \circ \cdots \circ G_2(z_{1,3-h}) \\ \vdots \\ G_n\circ G_{n-1} \circ \cdots \circ G_2(z_{1,h}) \end{bmatrix}, \end{equation}
in which the $2h \times 2h$ matrix is zero on and above the main diagonal and has all other entries equal to $1$.  When $n=1$ there is nothing to prove.  When $n \geq 2$, by inductive hypothesis the right side is equal to
\begin{align*}
\sigma_{n-1} \begin{bmatrix} 0 & 0 & 0 & \dots & 0 \\ 1 & 0 & 0 & \dots & 0 \\ 1 & 1 & 0 & \dots & 0 \\ \vdots & \vdots & \vdots & \ddots & \vdots \\ 1 & 1 & 1 & \dots & 0 \end{bmatrix}
\begin{bmatrix} G_n(z_{n-1,1-h}) \\ G_n(z_{n-1,2-h}) \\ G_n(z_{n-1,3-h}) \\ \vdots \\ G_n(z_{n-1,h}) \end{bmatrix},
\end{align*}
and by Lemma~\ref{lem:image-zni}, the row on the right side corresponding to $\sigma_n z_{n,i-h}$ on the left is then
\begin{align*}
\sigma_{n-1} G_n\left(\sum_{j=1-h}^{i-1-h} z_{n-1,j}\right) &=
\sigma_{n-1}\cdot \frac{\epsilon_{n-1}}{2}\sum_{j=1-h}^{i-1-h} (z_{n,j} - z_{n,j+1}) \\
&= -\sigma_n (z_{n,1-h} - z_{n,i-h})
\end{align*}
since the sum on the right telescopes.  But $z_{n,1-h} = 0$ by \eqref{eq:z-adjunction-bounds}, since $n \geq 2$, so the right side is equal to $\sigma_nz_{n,i-h}$ as desired and this completes the induction.

Having proved \eqref{eq:zni-recurrence}, we note that the $2h\times 2h$ matrix $M$ on the right side is the adjacency matrix of a directed graph $\Gamma$ with vertices labeled $1,\dots,2h$ and a single edge $i \to j$ if and only if $i > j$.  Then the $(i,j)$ entry of $M^{n-1}$ counts paths in $\Gamma$ of the form
\[ i \to a_{n-2} \to a_{n-3} \to \dots \to a_1 \to j, \]
of total length $n-1$, and these are in bijection with $(n-2)$-element subsets of $\{j+1,j+2,\dots,i-1\}$, so there are $\binom{i-j-1}{n-2}$ of them.  (This number is zero whenever $i-j-1 < n-2$.)  Applying this to \eqref{eq:zni-recurrence}, we conclude for $1\leq i \leq 2h$ that
\[ \sigma_n z_{n,i-h} = G_n\circ\dots\circ G_2\left( \sum_{j=1}^{2h}\binom{i-j-1}{n-2} z_{1,j-h} \right). \]
And if we sum over $1 \leq i \leq 2h$ then we see that
\begin{align*}
\sigma_n I^\#(X_n,\nu_n)(\bfone) &= \sigma_n \sum_{i=1}^{2h} z_{n,i-h} \\
&= G_n \circ \dots \circ G_2 \left( \sum_{j=1}^{2h} \left(\sum_{i=1}^{2h} \binom{i-j-1}{n-2}\right) z_{1,j-h}\right) \\
&= G_n \circ \dots \circ G_2 \left( \sum_{j=1}^{2h} \left(\sum_{\ell=0}^{2h-j-1} \binom{\ell}{n-2}\right) z_{1,j-h}\right) \\
&= G_n \circ \dots \circ G_2 \left( \sum_{j=1}^{2h} \binom{2h-j}{n-1} z_{1,j-h} \right).
\end{align*}
Here we achieved the third equality by discarding the summands where $i \leq j$, since then $\binom{i-j-1}{n-2}=0$, and substituting $\ell = i-j-1$.  The last equality is an application of the identity
\begin{equation} \label{eq:sum-pt-diagonal}
\sum_{\ell=0}^m \binom{\ell}{k} = \binom{m+1}{k+1},
\end{equation}
which is easily proved by induction on $m$ via the identity $\binom{m+1}{k+1}+\binom{m+1}{k}=\binom{m+2}{k+1}$.  Noting that $2h-j = h - (j-h)$, and hence each $z_{1,i}$ on the right has coefficient $\binom{h-i}{n-1}$, completes the proof.
\end{proof}

\begin{proposition} \label{prop:X1-coefficients}
Fix a knot $K \subset S^3$ with $\cinvt(K) > 0$, and pick an integer $h \geq g(K)-1$.  For $1 \leq n \leq 2h+1$, we define elements
\[ x_n = \sum_{i=-h}^{h-n} c_{n,i} \cdot (-1)^i y_i \]
of $I^\#(S^3_0(K),\mu)$, where the coefficients $c_{n,i} \in \Z$ satisfy
\[ c_{1,i} = \begin{cases} 0, & i\equiv h\pmod{2} \\ 1, & i \not\equiv h \pmod{2}, \end{cases} \]
and for all $n\geq 2$ and $-h \leq i \leq h$,
\[ c_{n,i} = \sum_{k \geq 1 } \binom{(h-i)-2k}{n-2}. \]
(In this sum we interpret the range of summation to mean all $k\geq 1$ such that $(h-i)-2k \geq 0$, here and throughout the remainder of this section.)  Then we have
\[ (G_n \circ G_{n-1} \circ \dots \circ \tilde{G}_1)(x_n) = \pm\frac{1}{2^n} I^\#(X_n,\nu_n)(\bfone) \]
for all $n = 1,2,\dots,2h+1$.
\end{proposition}

\begin{proof}
We begin by expressing $z_{1,i}$ in terms of the various $y_i$.  By Lemma~\ref{lem:image-zni} we have
\begin{align*}
\tilde{G}_1\left(\sum_{i=-h}^{j-1} y_i\right)
&= \frac{\epsilon_0}{2} \sum_{i=-h}^{j-1} (-1)^i (z_{1,i} + z_{1,i+1}) \\
&= \frac{\epsilon_0}{2} \left((-1)^{j-1}z_{1,j} + (-1)^{-h}z_{1,-h} + \sum_{i=1-h}^{j-1} ((-1)^i + (-1)^{i-1})z_{1,i}\right) \\
&= (-1)^{j-1} \cdot \frac{\epsilon_0}{2} z_{1,j},
\end{align*}
since $z_{1,-h} = 0$ by \eqref{eq:z-adjunction-bounds}.  Combining this with Lemma~\ref{lem:X1-in-terms-of-z1}, we have
\begin{align*}
\frac{\epsilon_0\sigma_n}{2} I^\#(X_n,\nu_n)(\bfone)
&= G_n\circ\dots\circ G_2\left( \sum_{i=1-h}^{h-n+1} \binom{h-i}{n-1}\cdot\frac{\epsilon_0}{2}z_{1,i} \right) \\
&= G_n\circ\dots\circ \tilde{G}_1\left( \sum_{i=1-h}^{h-n+1} \binom{h-i}{n-1}\cdot (-1)^{i-1}\left(\sum_{j=-h}^{i-1}y_j \right)\right) \\
&= G_n\circ\dots\circ \tilde{G}_1\left( \sum_{j=-h}^{h-n} \left( \sum_{i=j+1}^{h-n+1} (-1)^{i-j-1}\binom{h-i}{n-1}
\right) (-1)^j y_j \right).
\end{align*}

We can simplify the coefficient of $(-1)^j y_j$ somewhat: it is equal to
\[ \binom{h-j-1}{n-1} - \binom{h-j-2}{n-1} + \binom{h-j-3}{n-1} - \dots - (-1)^{h-j-(n-1)}\binom{n-1}{n-1}. \]
If $n=1$ then this sum has the form $1-1+1-1+\dots$, which is zero if $h-j-(n-1)=h-j$ is even and one if it is odd.  If instead $n \geq 2$ then we apply the identity $\binom{\ell}{n-1} - \binom{\ell-1}{n-1} = \binom{\ell-1}{n-2}$ to pairs of consecutive terms, and the sum becomes
\[ \binom{h-j-2}{n-2} + \binom{h-j-4}{n-2} + \binom{h-j-6}{n-2} + \dots + \binom{n-1}{n-2} \]
if $h-j-(n-1)$ is even, while if $h-j-(n-1)$ is odd it becomes
\begin{multline*}
\left[\binom{h-j-2}{n-2} + \binom{h-j-4}{n-2} + \binom{h-j-6}{n-2} + \dots + \binom{n}{n-2}\right] + \binom{n-1}{n-1} \\
= \binom{h-j-2}{n-2} + \binom{h-j-4}{n-2} + \binom{h-j-6}{n-2} + \dots + \binom{n}{n-2} + \binom{n-2}{n-2}.
\end{multline*}
Thus in either case the coefficient of $(-1)^jy_j$ is
\[ \sum_{k \geq 1} \binom{h-j-2k}{n-2}, \]
completing the proof.
\end{proof}

We now recall from Proposition~\ref{prop:cinvt-count-y_i} that if $\cinvt(K) > 0$, then $\cinvt(K)$ is equal to the number of nonzero elements
\[ y_i = I^\#(X_0,\tilde{\nu}_0;s_{0,i})(\bfone) \in I^\#(S^3_0(K),\mu), \]
and that by Lemma~\ref{lem:conjugation-y_i}, a fixed element $y_i$ is nonzero if and only if $y_{-i}$ is nonzero.  Thus if $\cinvt(K)$ is positive and $y_0 = 0$, then we have $\cinvt(K) = 2k$ for some $k > 0$, and there are integers
\[ 0 < i_1 < i_2 < \dots < i_k \leq g(K)-1 \]
such that
\[ y_i \neq 0 \text{ if and only if } i = \pm i_j \text{ for some } j. \]

\begin{proposition} \label{prop:M-invertible}
Suppose that $\cinvt(K) > 0$ and $y_0 = 0$. Fix $h \geq g(K)-1$, and let
\[ 0 < i_1 < i_2 < \dots < i_k \leq h \]
be the sequence of integers $i$ such that $y_i = I^\#(X_0,\tilde\nu_0;s_{0,i})(\bfone)$ is nonzero if and only if $i=\pm i_j$ for some $j$.  Then the $2k \times 2k$ matrix
\[ M_{i_1,\dots,i_k;h} = \begin{bmatrix}
c_{0,-i_k} & \dots & c_{0,-i_2} & c_{0,-i_1} & c_{0,i_1} & c_{0,i_2} & \dots & c_{0,i_k} \\
c_{1,-i_k} & \dots & c_{1,-i_2} & c_{1,-i_1} & c_{1,i_1} & c_{1,i_2} & \dots & c_{1,i_k} \\
c_{2,-i_k} & \dots & c_{2,-i_2} & c_{2,-i_1} & c_{2,i_1} & c_{2,i_2} & \dots & c_{2,i_k} \\
c_{3,-i_k} & \dots & c_{3,-i_2} & c_{3,-i_1} & c_{3,i_1} & c_{3,i_2} & \dots & c_{3,i_k} \\
\vdots & \ddots & \vdots & \vdots & \vdots & \vdots & \ddots & \vdots \\
c_{2k-1,-i_k} & \dots & c_{2k-1,-i_2} & c_{2k-1,-i_1} & c_{2k-1,i_1} & c_{2k-1,i_2} & \dots & c_{2k-1,i_k} \end{bmatrix} \]
is invertible, where $c_{0,j} = (-1)^j$ and 
\begin{align*}
c_{1,j} &= \begin{cases} 0, & j\equiv h\pmod{2} \\ 1, & j\not\equiv h\pmod{2}, \end{cases} &
c_{n,j} &= \sum_{k \geq 1} \binom{(h-j)-2k}{n-2} \text{ for all } n\geq 2
\end{align*}
as in Proposition~\ref{prop:X1-coefficients}.
\end{proposition}

\begin{proof}
By the definition of $\cinvt(K) = 2k$, each of the maps
\[ I^\#(X_0,\tilde\nu_0) \quad\text{ and }\quad I^\#(X_n,\nu_n), \quad 1 \leq n \leq 2k-1 \]
is injective.  Following Proposition~\ref{prop:X1-coefficients}, we let
\begin{equation} \label{eq:xn-sum-definition}
x_n = \sum_{i=-h}^{h-n} c_{n,i} \cdot (-1)^iy_i \in I^\#(S^3_0(K),\mu)
\end{equation}
for $0 \leq n \leq 2k-1$, and it follows for $n \geq 1$ that the corresponding elements
\[ I^\#(X_n,\nu_n)(\bfone) = \begin{cases} 
x_0, & n=0 \\
\pm 2^n \cdot (G_n \circ G_{n-1} \circ \dots \circ \tilde{G}_1)(x_n), & 1 \leq n \leq 2k-1 \end{cases} \]
are all nonzero.  (For $n=0$ we interpret the left side as $I^\#(X_0,\tilde\nu_0)(\bfone)$.)

We claim that the elements $x_0,\dots,x_{2k-1}$ in \eqref{eq:xn-sum-definition} must all be linearly independent.  Assuming otherwise, there is some positive $j \leq 2k-1$ for which we can find a nontrivial relation of the form
\[ x_j = \sum_{i=0}^{j-1} d_i x_i. \]
But then we must have
\begin{align*}
I^\#(X_j,\nu_j)(\bfone) &= \pm 2^j \cdot (G_j \circ G_{j-1} \circ \dots \circ \tilde{G}_1) \left(\sum_{i=0}^{j-1} d_i x_i\right) \\
&= \pm \sum_{i=0}^{j-1} 2^{j-i}d_i\cdot (G_j \circ \dots \circ G_{i+1})\left(2^i\cdot G_i\circ\dots\circ \tilde{G}_1(x_i)\right) \\
&= \pm \sum_{i=0}^{j-1} 2^{j-i}d_i\cdot (G_j \circ \dots \circ G_{i+1})\left(\pm I^\#(X_i,\nu_i)(\bfone)\right).
\end{align*}
(In the above equations, when $i=0$ we should again interpret $(X_i,\nu_i)$ as $(X_0,\tilde\nu_0)$, and also $G_{i+1}$ as $\tilde{G}_1$.)  We know for each $i$ that $G_{i+1} \circ I^\#(X_i,\nu_i) = 0$ by the surgery exact triangle, so we have $I^\#(X_j,\nu_j)(\bfone) = 0$, contradiction.

Now by assumption we have $y_j = 0$ for $j\not\in\{\pm i_1,\dots,\pm i_k\}$, so from \eqref{eq:xn-sum-definition} we see that 
\[ \begin{bmatrix} x_0 \\ x_1 \\ \vdots \\ x_{2k-1} \end{bmatrix} = M_{i_1,\dots,i_k;h}
\begin{bmatrix} (-1)^{-i_k}y_{-i_k} \\ \vdots \\ (-1)^{-i_1}y_{-i_1} \\ (-1)^{i_1}y_{i_1} \\ \vdots \\ (-1)^{i_k} y_{i_k} \end{bmatrix}. \]
The linear independence of $x_0,\dots,x_{2k-1}$ then implies that $M_{i_1,\dots,i_k;h}$ is invertible, exactly as claimed.
\end{proof}

\subsection{Some linear algebra} \label{ssec:linalg}

In Proposition~\ref{prop:M-invertible} we associated to every $k \geq 1$ and sequence of integers
\[ 0 < i_1 < i_2 < \dots < i_k \leq h \]
a $2k \times 2k$ matrix
\[ M_{i_1,\dots,i_k;h} = \begin{bmatrix}
c_{0,-i_k} & \dots & c_{0,-i_1} & c_{0,i_1} & \dots & c_{0,i_k} \\
c_{1,-i_k} & \dots & c_{0,-i_1} & c_{0,i_1} & \dots & c_{0,i_k} \\
\vdots & \ddots & \vdots & \vdots & \ddots & \vdots \\
c_{2k-1,-i_k} & \dots & c_{2k-1,-i_1} & c_{2k-1,i_1} & \dots & c_{2k-1,i_k} \\
\end{bmatrix}. \]
In this subsection we will prove by induction on $k$ that $M_{i_1,\dots,i_k;h}$ is never invertible.  To do so, we will define a related $2k \times k$ integer matrix
\begin{align} \label{eq:N-matrix-def}
N_{i_1,\dots,i_k;h} &= \begin{bmatrix} d_{0,i_1} & \dots & d_{0,i_k} \\ d_{1,i_1} & \dots & d_{1,i_k} \\ \vdots & \ddots & \vdots \\ d_{2k-1,i_1} & \dots & d_{2k-1,i_k} \end{bmatrix}, &
d_{j,i} &= c_{j,-i} - c_{j,i}
\end{align}
motivated by the idea that this should contain the same information as the matrix $M_{i_1,\dots,i_k;h}$ after we account for the conjugation symmetries of Section~\ref{sec:conjugation}.  Indeed, we observe that
\begin{equation} \label{eq:N-kernel-implication}
N_{i_1,\dots,i_k;h} \begin{bmatrix} x_1 \\ \vdots \\ x_k \end{bmatrix} = 0
\quad\Longrightarrow\quad
M_{i_1,\dots,i_k;h} \begin{bmatrix} x_k \\ \vdots \\ x_1 \\ -x_1 \\ \vdots \\ -x_k \end{bmatrix} = 0,
\end{equation}
so it will suffice to find a nonzero vector $\begin{bmatrix} x_1 & \cdots & x_k \end{bmatrix}^T$ in the kernel of $N_{i_1,\dots,i_k;h}$.  This will ultimately come from the fact that for fixed $j$, the various $d_{j,i}$ are linear combinations of binomial coefficients of bounded degree, which turn out to be polynomials in $i$; by construction these polynomials are also odd, so then each row of $N_{i_1,\dots,i_k;h}$ is a linear combination of relatively few row vectors of the form $\begin{bmatrix} i_1^e & \cdots & i_k^e \end{bmatrix}$, with $e$ an odd integer, and this leads us to an upper bound on the rank of $N_{i_1,\dots,i_k;h}$ and hence to an element of its kernel.

Combining the definition of $c_{j,i}$ from Proposition~\ref{prop:M-invertible} with \eqref{eq:N-matrix-def}, we see that if $-h \leq i \leq h$ then we have
\begin{align*}
d_{0,i} &= (-1)^{-i} - (-1)^{i} = 0 \\
d_{1,i} &= c_{1,-i} - c_{1,i} = 0
\end{align*}
since $\pm i \equiv h\pmod{2}$ if and only if $i \equiv h \pmod{2}$.  It follows immediately that when $k=1$ we have
\[ N_{i_1;h} = \begin{bmatrix} d_{0,i_1} \\ d_{1,i_1} \end{bmatrix} = \begin{bmatrix} 0 \\ 0 \end{bmatrix}, \]
and so $N_{i_1;h} \begin{bmatrix} 1 \end{bmatrix} = 0$.  For larger values of $j$, we begin with the following observation about the entries of $N_{i_1,\dots,i_k;h}$.

\begin{lemma} \label{lem:N-matrix-polynomials}
For all $0 \leq j \leq 2h+1$, there is an odd polynomial $p_j(t) \in \Q[t]$ of degree at most $\max(j-1,0)$, depending implicitly on $h$, such that
\[ p_j(i) = d_{j,i} \]
for all integers $i$ such that $1 \leq i \leq h$.  When $j$ is even and at least $2$, the degree of $p_j(t)$ is exactly $j-1$.
\end{lemma}

\begin{proof}
For $j=0,1$ we have just seen that we can take $p_j(t) = 0$, so we assume from now on that $j \geq 2$.
We extend $d_{j,i}$ to the range $-h \leq i \leq 0$ by setting $d_{j,i} = c_{j,-i} - c_{j,i}$ for all such $i$, so that $d_{j,-i} = -d_{j,i}$ for $-h \leq i \leq h$.  By definition, when $1 \leq i \leq h$ we have
\begin{align*}
d_{j,i} &= c_{j,-i} - c_{j,i} \\
&= \left(\sum_{k\geq 1} \binom{h+i-2k}{j-2}\right) - \left(\sum_{k\geq 1} \binom{h-i-2k}{j-2}\right) \\
&= \binom{h+i-2}{j-2} + \binom{h+i-4}{j-2} + \binom{h+i-6}{j-2} + \dots + \binom{h-i}{j-2}.
\end{align*}
If $j=2$ then this means that $d_{2,i} = i$ for $1 \leq i \leq h$ and hence for $-h \leq i \leq h$ as well, so we can take $p_2(t) = t$.

Now suppose for some $j \geq 3$ that we have proved the lemma for all smaller values of $j$.  We will prove the lemma for $j$ as well by first showing that
\begin{equation} \label{eq:dji-differences}
d_{j,i+1}-d_{j,i} = d_{j-1,i} + \binom{h-i-1}{j-2}
\end{equation}
for all integers $i$ in the range $-h \leq i < h$.  Indeed, when $0 \leq i < h$, we compute that
\begin{align*}
d_{j,i+1}-d_{j,i} &= \left[ \binom{h+i-1}{j-2} - \binom{h+i-2}{j-2} \right] + \left[ \binom{h+i-3}{j-2} - \binom{h+i-4}{j-2} \right] \\
&\qquad + \dots + \left[ \binom{h-i+1}{j-2} - \binom{h-i}{j-2} \right] + \binom{h-i-1}{j-2} \\
&= \left[ \binom{h+i-2}{j-3} + \binom{h+i-4}{j-3} + \dots + \binom{h-i}{j-3} \right] + \binom{h-i-1}{j-2} \\
&= d_{j-1,i} + \binom{h-i-1}{j-2},
\end{align*}
as desired.

Similarly, when $-h \leq i < 0$ we use the $|i|-1$ case of this computation to show that
\begin{equation} \label{eq:d-difference-i-negative}
d_{j,i+1}-d_{j,i} = d_{j,|i|} - d_{j,|i|-1} = d_{j-1,|i|-1} + \binom{h-|i|}{j-2}.
\end{equation}
The right side does not quite have the same form as before, but we can check that
\begin{align*}
d_{j-1,|i|-1} + d_{j-1,|i|} &= \left[\binom{h+|i|-3}{j-3} + \binom{h+|i|-5}{j-3} + \dots + \binom{h-|i|+1}{j-3}\right] \\
&\qquad + \left[\binom{h+|i|-2}{j-3} + \binom{h+|i|-4}{j-3} + \dots + \binom{h-|i|}{j-3}\right] \\
&= \sum_{\ell=0}^{h+|i|-2} \binom{\ell}{j-3} - \sum_{\ell=0}^{h-|i|-1} \binom{\ell}{j-3} \\
&= \binom{h+|i|-1}{j-2} - \binom{h-|i|}{j-2},
\end{align*}
where we have again used the identity \eqref{eq:sum-pt-diagonal} in the last line.  This allows us to rewrite the right side of \eqref{eq:d-difference-i-negative} for $-h \leq i < 0$ as
\[ d_{j,i+1} - d_{j,i} = \binom{h+|i|-1}{j-2} - d_{j-1,|i|}, \]
and now since $|i|=-i$ and $d_{j-1,|i|} = -d_{j-1,i}$, this establishes \eqref{eq:dji-differences} for $-h \leq i < 0$ as well.

Now by hypothesis, for $-h \leq i \leq h$ we have
\[ d_{j-1,i} = p_{j-1}(i), \]
where $p_{j-1}$ is an odd polynomial of degree at most $j-2$, and if $j$ is even then $\deg p_{j-1} \leq j-3$ because odd polynomials must have odd degree.  Moreover,
\[ \binom{h-i-1}{j-2} = \frac{ ((h-1)-i)((h-2)-i) \cdots ((h-(j-2))-i) }{ (j-2)! } \]
is a polynomial in $i$ (depending implicitly on our choice of $h$) of degree exactly $j-2$.  Thus we can define the polynomial $p_j(t)$ to be the unique one satisfying $p_j(0) = 0$ and 
\begin{equation} \label{eq:d-difference-equation}
p_j(t+1) - p_j(t) = p_{j-1}(t) +  \frac{1}{(j-2)!}\prod_{\ell=1}^{j-2} ((h-\ell)-t), \quad t\in\Z,
\end{equation}
and by \eqref{eq:dji-differences} we will have $p_j(i) = d_{j,i}$ for $-h \leq i \leq h$.  The successive differences $p_j(t+1)-p_j(t)$ are polynomials of degree at most $j-2$, with equality when $j$ is even since the $\binom{h-i-1}{j-2}$ term on the right side of \eqref{eq:d-difference-equation} is the unique term on that side with degree exactly $j-2$.  We conclude that $\deg p_j(t) \leq j-1$, and that equality holds when $j$ is even.

It remains to be seen that $p_j(t)$ is also odd.  To do so, we note that the polynomial
\[ p_j(t) + p_j(-t) \]
has degree at most $j-1$, and that
\[ p_j(i) + p_j(-i) = d_{j,i} + d_{j,-i} = 0 \quad \text{for } i=-h,-h+1,-h+2,\dots,h. \]
Since this polynomial has degree at most $j-1$ and $2h+1 \geq j$ distinct zeroes, it must be identically zero, and therefore $p_j(-t) = p_j(t)$ as desired.  This completes the proof by induction.
\end{proof}

\begin{proposition} \label{prop:N-kernel}
For any sequence of integers $0 < i_1 < i_2 < \dots < i_k \leq h$, there is a nonzero vector $v \in \Z^k$ such that
\[ (N_{i_1,\dots,i_k;h})v = 0. \]
\end{proposition}

\begin{proof}
By Lemma~\ref{lem:N-matrix-polynomials}, the rows of $N_{i_1,\dots,i_k;h}$ have the form
\[ \begin{bmatrix} d_{j,i_1} & d_{j,i_2} & \dots & d_{j,i_k} \end{bmatrix} = \begin{bmatrix} p_j(i_1) & p_j(i_2) & \dots & p_j(i_k) \end{bmatrix} \]
for $0 \leq j \leq 2k-1$, where each $p_j(t) \in \Q[t]$ is an odd polynomial of degree at most $\max(j-1,0) \leq 2k-2$.  Since $p_j(t)$ is a $\Q$-linear combination of the $k-1$ monomials $t, t^3, t^5, \dots, t^{2k-3}$, it follows that each row of $N_{i_1,\dots,i_k;h}$ is a $\Q$-linear combination of the $k-1$ row vectors
\[ \begin{bmatrix} i_1^e & i_2^e & \dots & i_k^e \end{bmatrix}, \quad e=1,3,5,\dots,2k-3. \]
This means that the row space of $N_{i_1,\dots,i_k;h}$ has dimension at most $k-1$, so
\[ \rank N_{i_1,\dots,i_k;h} \leq k-1 \]
and since $N_{i_1,\dots,i_k;h}$ is a $2k \times k$ integer matrix, its kernel must have dimension at least 1 and a basis of integer vectors.  
\end{proof}

We can now use this computation to prove Proposition~\ref{prop:y0-nonzero}.

\begin{proof}[Proof of Proposition~\ref{prop:y0-nonzero}]
If $\cinvt(K) = 0$ but $I^\#(X_0,\tilde\nu_0)$ is injective, then Remark~\ref{rmk:W-shaped} says that exactly one $y_i$ is nonzero, and by Lemma~\ref{lem:conjugation-y_i} it must be $y_0$.

Now suppose that $\cinvt(K)$ is positive but that $y_0 = 0$.  We let
\[ 0 < i_1 < i_2 < \dots < i_k \leq g(K)-1 \]
be the sequence of integers such that $y_i \neq 0$ if and only if $i = \pm i_j$ for some $j$.  Fixing some integer $h \geq g(K)-1$, Proposition~\ref{prop:M-invertible} says that the matrix $M_{i_1,\dots,i_k;h}$ is invertible, and then the matrix $N_{i_1,\dots,i_k;h}$ defined in \eqref{eq:N-matrix-def} must define an injective map $\Z^k \to \Z^{2k}$, by the observation \eqref{eq:N-kernel-implication}.  Proposition~\ref{prop:N-kernel} says that this map has nontrivial kernel, however, so we have a contradiction.
\end{proof}

This completes the proof of Theorem~\ref{thm:nu-odd}. \qed

\section{Framed instanton homology of zero-surgery} \label{sec:zero-surgery}

If $K \subset S^3$ is a knot satisfying $\cinvt(K) \neq 0$, then Theorem~\ref{thm:dim-surgery} applies to zero-framed surgery on $K$ with any choice of bundle: that is, we have
\[ \cinvt(K) \neq 0 \ \Longrightarrow\ \dim I^\#(S^3_0(K)) = \dim I^\#(S^3_0(K),\mu) = r_0(K) + |\cinvt(K)|. \]
(Here we recall that $I^\#(S^3_0(K))$ is defined in terms of the trivial $U(2)$-bundle over $S^3_0(K)$, while $I^\#(S^3_0(K))$ uses a bundle $E \to S^3_0(K)$ with $c_1(E)$ Poincar\'e dual to $\mu$.)  This independence is explicitly stated in \cite[Proposition~3.3]{bs-concordance}; it follows from the $p/q=\pm1$ cases of Theorem~\ref{thm:dim-surgery} and the surgery exact triangle.  But when $\cinvt(K) = 0$, the framed instanton homology of $S^3_0(K)$ need not be equal to $r_0(K)+|\cinvt(K)|$.  Here we use the conjugation symmetry of Theorem~\ref{thm:conjugation} to investigate this remaining case. 

\begin{theorem} \label{thm:nu-zero-surgery}
Let $Y$ be an integer homology sphere with $\dim I^\#(Y) = 1$, and let $K \subset Y$ be a knot with meridian $\mu$ such that
\[ \dim I^\#(Y_{-1}(K)) = \dim I^\#(Y_1(K)) = d. \]
Then $\dim I^\#(Y_0(K))$ and $\dim I^\#(Y_0(K),\mu)$ are equal to $d-1$ and $d+1$ in some order.
\end{theorem}

As noted in Remark \ref{rmk:splice}, we have applied Theorem~\ref{thm:nu-zero-surgery} elsewhere, as \cite[Proposition~3.2]{bs-splice}, in order to prove that many integer homology spheres built by splicing nontrivial knot complements cannot be instanton L-spaces. 

\begin{corollary} \label{cor:nu-0-twisted}
A knot $K \subset S^3$ has $\cinvt(K)=0$ if and only if
\[ \dim I^\#(S^3_0(K)) \neq \dim I^\#(S^3_0(K),\mu), \]
in which case these dimensions are $r_0(K)$ and $r_0(K)+2$ in some order.
\end{corollary}

\begin{proof}
In the case $\cinvt(K) \neq 0$ this is part of \cite[Proposition~3.3]{bs-concordance}.  When $\cinvt(K)=0$, we let $d=\dim I^\#(S^3_{\pm1}(K))$, which is $r_0(K)+1$ by Theorem~\ref{thm:dim-surgery}.  Then Theorem~\ref{thm:nu-zero-surgery} says that $\dim I^\#(S^3_0(K))$ and $\dim I^\#(S^3_0(K),\mu)$ are equal to $d-1 = r_0(K)$ and $d+1 = r_0(K)+2$ in some order.
\end{proof}

\begin{remark}
Theorem~\ref{thm:nu-zero-surgery} allows for $\cinvt(K)$ and $r_0(K)$ to be defined for knots not just in $S^3$ but in any integer homology sphere $Y$ with $\dim I^\#(Y) = 1$.  These will still satisfy the formula
\[ \dim I^\#(Y_{p/q}(K)) = q\cdot r_0(K) + |p - q\cinvt(K)|, \]
just as in Theorem~\ref{thm:dim-surgery}, though we do not claim that $\cinvt(K)$ is a quasi-morphism for knots in manifolds other than $S^3$.  The key is to show (following arguments of \cite{bs-concordance}) that the sequence \[\left(\dim I^\#(Y_n(K))\right)_{n\in\Z}\] either has a unique local minimum at some $n$, or it has two at $n=\pm1$; in the latter case we replace $I^\#(Y_0(K))$ with $I^\#(Y_0(K),\mu)$, and by Theorem~\ref{thm:nu-zero-surgery} it now has a unique minimum at $n=0$.  We take $\cinvt(K)$ to be that value of $n$, and proceed as in \cite{bs-concordance}.  This avoids the definition for $K\subset S^3$ in \cite{bs-concordance} by studying surgeries on both $K$ and its mirror, since $\mirror{K} \subset Y$ does not make sense in a general $Y$.
\end{remark}

We now prove Theorem~\ref{thm:nu-zero-surgery} in several steps.

\begin{lemma} \label{lem:nu-zero-surgery-not-d+1}
Let $(Y,K)$ be as in Theorem~\ref{thm:nu-zero-surgery}, with $d = \dim I^\#(Y_1(K))$.  Then $\dim I^\#(Y_0(K))$ and $\dim I^\#(Y_0(K),\mu)$ cannot both be $d+1$.
\end{lemma}

\begin{proof}
This is essentially already contained in the proof of \cite[Proposition~3.3]{bs-concordance}.  We consider two surgery exact triangles involving $I^\#(Y_0(K))$, both special cases of \eqref{eq:surgery-triangle}:
\begin{align}
\dots &\to I^\#(Y) \xrightarrow{I^\#(X_0,\nu_0)} I^\#(Y_0(K)) \to I^\#(Y_1(K)) \to \dots \label{eq:d+1-1}\\
\dots &\to I^\#(Y) \to I^\#(Y_{-1}(K)) \to I^\#(Y_0(K)) \xrightarrow{I^\#(Z_0,\zeta_0)} \dots \label{eq:d+1-2}
\end{align}
as well as two more involving $I^\#(Y_0(K),\mu)$, the first of which is \eqref{eq:surgery-triangle-0} and the second of which is derived similarly from \eqref{eq:surgery-triangle-general}:
\begin{align}
\dots &\to I^\#(Y) \xrightarrow{I^\#(X_0,\tilde\nu_0)} I^\#(Y_0(K),\mu) \to I^\#(Y_1(K)) \to \dots \label{eq:d+1-3} \\
\dots &\to I^\#(Y) \to I^\#(Y_{-1}(K)) \to I^\#(Y_0(K),\mu) \xrightarrow{I^\#(Z_0,\tilde\zeta_0)} \dots. \label{eq:d+1-4}
\end{align}
Exactly as in \cite[Proposition~3.3]{bs-concordance}, the composite cobordisms
\[ I^\#(X_0,\tilde\nu_0) \circ I^\#(Z_0,\zeta_0): I^\#(Y_0(K)) \to I^\#(Y_0(K),\mu) \]
and
\[ I^\#(X_0,\nu_0) \circ I^\#(Z_0,\tilde\zeta_0): I^\#(Y_0(K),\mu) \to I^\#(Y_0(K)) \]
are both identically zero, because they represent cobordisms containing an embedded 2-sphere equipped with a nontrivial bundle.  But if $\dim I^\#(Y_0(K))=d+1$ then $I^\#(X_0,\nu_0)$ is injective by the exactness of \eqref{eq:d+1-1}, so now $I^\#(Z_0,\tilde\zeta_0)$ must be zero, whence by \eqref{eq:d+1-4} we have $\dim I^\#(Y_0(K),\mu) = d-1$.  Similarly, if $\dim I^\#(Y_0(K),\mu)=d+1$ then $I^\#(X_0,\tilde\nu_0)$ is injective by \eqref{eq:d+1-3}, so $I^\#(Z_0,\zeta_0) = 0$, and then \eqref{eq:d+1-2} says that $\dim I^\#(Y_0(K)) = d-1$.
\end{proof}

\begin{lemma} \label{lem:nu-zero-surgery-yi-zero}
Let $(Y,K)$ be as in Theorem~\ref{thm:nu-zero-surgery}, with $d = \dim I^\#(Y_1(K))$, and fix a generator $\bfone \in I^\#(Y)$.  If $\dim I^\#(Y_0(K)) = d - 1$, then the elements
\[ y_i' := I^\#(X_0,\nu_0;s_i)(\bfone) \in I^\#(Y_0(K)) \]
defined in \eqref{eq:tilde-y_i} are all zero.  Similarly, if $\dim I^\#(Y_0(K),\mu) = d - 1$, then the elements
\[ y_i := I^\#(X_0,\tilde\nu_0;s_i)(\bfone) \in I^\#(Y_0(K),\mu) \]
defined in \eqref{eq:y_i} are all zero.
\end{lemma}

\begin{proof}
Suppose that $\dim I^\#(Y_0(K)) = d-1$.  Then we examine \eqref{eq:d+1-1} to see that $I^\#(X_0,\nu_0) = 0$ for dimension reasons.  In the decomposition
\[ I^\#(X_0,\nu_0) = \sum_{i \in \Z} I^\#(X_0,\nu_0;s_i), \]
the $s_i$ each satisfy $s_i([\hat\Sigma]) = 2i$ as defined in \eqref{eq:def-si}, and restrict to $t_i: H_2(Y_0(K);\Z) \to \Z$ as in \eqref{eq:def-ti}.  Then each
\[ I^\#(X_0,\nu_0;s_i): I^\#(Y) \to I^\#(Y_0(K); t_i) \]
has its image in a different eigenspace of $I^\#(Y_0(K))$, and since their sum is zero, they must all individually be zero as well.  (This fact is the reason we can only prove Theorem~\ref{thm:nu-zero-surgery} for zero-surgery: for other surgeries, the eigenspace decomposition is trivial.)  The elements $y_i'$ lie in the images of these maps, so they are all zero.

In the case where $\dim I^\#(Y_0(K),\mu) = d-1$, we apply the same argument using the triangle \eqref{eq:d+1-3} to see that $I^\#(X_0,\tilde\nu_0;s_i)=0$ for all $i$, hence $y_i=0$ for all $i$ as well.
\end{proof}

For the next lemma, we examine the triangles \eqref{eq:d+1-2} and \eqref{eq:d+1-4} in more depth.  The first of these has the form
\begin{equation} \label{eq:-1-triangle-untwisted}
\dots \to I^\#(Y) \xrightarrow{I^\#(X_{-1},\nu_{-1})} I^\#(Y_{-1}(K)) \xrightarrow{I^\#(W_0,\omega_0)} I^\#(Y_0(K)) \to \dots,
\end{equation}
where if $[\Sigma_{-1}]$ is a generator of $H_2(X_{-1};\Z)$ built by capping off a Seifert surface for $K$ with the core of the 2-handle, then $\nu_{-1} \cdot [\Sigma_{-1}]$ is odd.  Similarly, the second one has the form
\begin{equation} \label{eq:-1-triangle-twisted}
\dots \to I^\#(Y) \xrightarrow{I^\#(X_{-1},\tilde\nu_{-1})} I^\#(Y_{-1}(K)) \xrightarrow{I^\#(W_0,\tilde\omega_0)} I^\#(Y_0(K),\mu) \to \dots.
\end{equation}
The surfaces decorating $X_{-1}$ in these triangles are related by $\tilde\nu_{-1} = \nu_{-1} + \Sigma_{-1}$ in $H_2(X_{-1})$; the intersection products $\nu_{-1}\cdot\Sigma_{-1}$ and $\tilde\nu_{-1}\cdot\Sigma_{-1}$ are odd and even respectively.  We also define homomorphisms
\[ s_{-1,i}: H_2(X_{-1};\Z) \to \Z \]
by $s_{-1,i}([\Sigma_{-1}]) = 2i+1$, exactly as in \eqref{eq:ts-homomorphisms}.

\begin{lemma} \label{lem:nu-zero-surgery-conjugate}
Take $(Y,K)$, $d = \dim I^\#(Y_1(K))$, and $\bfone \in I^\#(Y)$ as in Lemma~\ref{lem:nu-zero-surgery-yi-zero}.  If $\dim I^\#(Y_0(K)) = d-1$, then the elements
\[ z_{-1,i} := I^\#(X_{-1},\nu_{-1};s_{-1,i})(\bfone) \]
satisfy $z_{-1,i} = z_{-1,-1-i}$ for all $i$.  If $\dim I^\#(Y_0(K),\mu)=d-1$, then the elements
\[ \tilde{z}_{-1,i} := I^\#(X_{-1},\tilde\nu_{-1};s_{-1,i})(\bfone) \]
satisfy $\tilde{z}_{-1,i} = \tilde{z}_{-1,-1-i}$ for all $i$.
\end{lemma}

\begin{proof}
Supposing first that $\dim I^\#(Y_0(K)) = d-1$, we look at the triangle \eqref{eq:-1-triangle-untwisted}.  Repeating the argument of \cite[Proposition~7.1]{bs-lspace}, we write \[X_{-1} \cup_{Y_{-1}(K)} W_0 \cong X_0 \# \cptwo,\] by a diffeomorphism identifying $[\Sigma_{-1}]$ with $[\Sigma_0]-[E]$ where $E$ is the exceptional sphere on the right.  We now argue that $I^\#(W_0,\omega_0)$ sends $z_{-1,i}$ to a linear combination of various $y'_j$, all of which are zero.  Indeed, by the composition law for cobordism maps, we have
\begin{align*}
I^\#(W_0,\omega_0)\big(z_{-1,i}\big) &= \left(I^\#(W_0,\omega_0) \circ I^\#(X_{-1},\nu_{-1}; s_{-1,i})\right) (\bfone) \\
&= \sum_{u: H_2(W_0) \to \Z} \left(I^\#(W_0,\omega_0; u) \circ I^\#(X_{-1},\nu_{-1}; s_{-1,i})\right) (\bfone) \\
&= \sum_u \sum_{\substack{h:H_2(X_{-1}\cup W_0) \to \Z \\ h|_{X_{-1}}=s_{-1,i} \\ h|_{W_0}=u}} I^\#(X_{-1} \cup W_0; \nu_{-1} \cup \omega_0; h)(\bfone)
\end{align*}
Then every summand on the right has the form $I^\#(X_0\#\cptwo,\nu;h)(\bfone)$ for some $\nu$, and the blow-up formula and sign change formula of \cite[Theorem~1.16]{bs-lspace} say that such a term is in turn a scalar multiple (possibly zero) of some
\[ I^\#(X_0,\nu_0; s_j)(\bfone) = y_j', \]
which is zero by Lemma~\ref{lem:nu-zero-surgery-yi-zero}.  Thus $I^\#(W_0,\omega_0)(z_{-1,i}) = 0$.

Now by the exactness of \eqref{eq:-1-triangle-untwisted}, the element $z_{-1,i}$ is in the kernel of $I^\#(W_0,\omega_0)$, so it is in the image of $I^\#(X_{-1},\nu_{-1})$ and we can write
\[ z_{-1,i} = a_i \cdot I^\#(X_{-1},\nu_{-1})(\bfone) \]
for some coefficients $a_i \in \Q$.  By definition, the maps $s_{-1,i}: H_2(X_{-1}) \to \Z$ satisfy $s_{-1,i}([\Sigma_{-1}]) = 2i+1$, where $[\Sigma_{-1}]$ generates $H_2(X_{-1}) \cong \Z$.  It follows that $-s_{-1,i} = s_{-1,-1-i}$, so by the conjugation symmetry of Theorem~\ref{thm:conjugation}, we have
\[ c_{k+d} \circ I^\#(X_{-1},\nu_{-1};s_{-1,i}) = I^\#(X_{-1},\nu_{-1};s_{-1,-1-i}) \circ c_k \]
where $d$ is the degree of $I^\#(X_{-1},\nu_{-1})$.  If we apply both sides to the element $\bfone \in I^\#(Y)$ and take $k$ to be the grading of $\bfone$, then $c_k(\bfone)=\bfone$, so that
\[ c_{k+d}(z_{-1,i}) = z_{-1,-1-i}. \]
But $z_{-1,i}$ is a scalar multiple of $I^\#(X_{-1},\nu_{-1})(\bfone)$, which is homogeneous of grading $k+d$ and hence fixed by $c_{k+d}$, so in fact
\[ z_{-1,i} = z_{-1,-1-i} \]
for all $i$.

The proof that if $\dim I^\#(Y_0(K),\mu) = d-1$ then $\tilde{z}_{-1,i} = \tilde{z}_{-1,-1-i}$ is nearly identical.  We use \eqref{eq:-1-triangle-twisted} instead of \eqref{eq:-1-triangle-untwisted} to show that $I^\#(W_0,\tilde\omega_0)(\tilde{z}_{-1,i}) = 0$, because it is a linear combination of the elements $y_j$ which are zero by Lemma~\ref{lem:nu-zero-surgery-yi-zero}.  So then $\tilde{z}_{-1,i}$ is a scalar multiple of $I^\#(X_{-1},\tilde\nu_{-1})(\bfone)$, and we apply the same conjugation symmetry argument to $I^\#(X_{-1},\tilde\nu_{-1})$ (instead of $I^\#(X_{-1},\nu_{-1})$) to get the desired conclusion.
\end{proof}

\begin{proof}[Proof of Theorem~\ref{thm:nu-zero-surgery}]
By assumption we have $\dim I^\#(Y_{-1}(K)) = \dim I^\#(Y_1(K)) = d$.  The triangles \eqref{eq:d+1-1} and \eqref{eq:d+1-3} show that $\dim I^\#(Y_0(K))$ and $\dim I^\#(Y_0(K),\mu)$ must each be either $d-1$ or $d+1$, and Lemma~\ref{lem:nu-zero-surgery-not-d+1} says that they are not both $d+1$, so it remains to be seen that they are not both $d-1$ either.

Supposing that $\dim I^\#(Y_0(K)) = d-1$, we showed in Lemma~\ref{lem:nu-zero-surgery-conjugate} that
\[ I^\#(X_{-1},\nu_{-1}; s_i) = I^\#(X_{-1},\nu_{-1}; s_{-1-i}) \]
for all $i$, since these maps $I^\#(Y) \to I^\#(Y_{-1}(K))$ are determined by their images $z_{-1,i}$ and $z_{-1,-1-i}$.  The sign change formula of \cite[Theorem~1.16]{bs-lspace} says for any $s$ and $\alpha \in H_2(X_{-1};\Z)$ that
\[ I^\#(X_{-1},\nu_{-1}+\alpha; s) = (-1)^{\frac{1}{2}(s(\alpha)+\alpha\cdot\alpha)+\nu_{-1}\cdot\alpha} \cdot I^\#(X_{-1},\nu_{-1};s). \]
Taking $\alpha = [\Sigma_{-1}]$, we have $s_{-1,j}(\alpha) = 2j+1$, $\alpha\cdot\alpha=-1$, $\nu_{-1}\cdot\alpha=-1$, and $\tilde\nu_{-1}=\nu_{-1}+\alpha$, so this simplifies to
\begin{equation} \label{eq:nu-1-change}
I^\#(X_{-1},\tilde\nu_{-1}; s_{-1,j}) = (-1)^{j-1} \cdot I^\#(X_{-1},\nu_{-1};s_{-1,j})
\end{equation}
for all $j$.  Thus
\begin{align*}
I^\#(X_{-1},\tilde\nu_{-1}; s_{-1,i}) &= (-1)^{i-1} \cdot I^\#(X_{-1},\nu_{-1};s_{-1,i}) \\
&= (-1)^{i-1} \cdot I^\#(X_{-1},\nu_{-1};s_{-1,-1-i}) \\
&= (-1)^{i-1} \cdot \left( (-1)^{-i} \cdot I^\#(X_{-1},\tilde\nu_{-1};s_{-1,-1-i}) \right) \\
&= - I^\#(X_{-1},\tilde\nu_{-1};s_{-1,-1-i}).
\end{align*}
Thus we can compute the map associated to $(X_{-1},\tilde\nu_{-1})$ by
\begin{align*}
I^\#(X_{-1},\tilde\nu_{-1}) &= \sum_{i\in\Z} I^\#(X_{-1},\tilde\nu_{-1}; s_{-1,i}) \\
&= \sum_{i \geq 0} \left(I^\#(X_{-1},\tilde\nu_{-1}; s_{-1,i}) + I^\#(X_{-1},\tilde\nu_{-1}; s_{-1,-1-i})\right) \\
&= \sum_{i \geq 0} 0 = 0.
\end{align*}
But this map fits precisely into the exact triangle \eqref{eq:-1-triangle-twisted}, so we conclude that
\[ \dim I^\#(Y_0(K),\mu) = \dim I^\#(Y_{-1}(K)) + \dim I^\#(Y) = d+1. \]

The case where $\dim I^\#(Y_0(K),\mu) = d-1$ is nearly identical.  Here Lemma~\ref{lem:nu-zero-surgery-conjugate} says that
\[ I^\#(X_{-1},\tilde\nu_{-1}; s_i) = I^\#(X_{-1},\tilde\nu_{-1}; s_{-1-i}) \]
for all $i$, so we repeat the argument after \eqref{eq:nu-1-change}, exchanging the roles of $\nu_{-1}$ and $\tilde\nu_{-1}$, to get
\[ I^\#(X_{-1},\nu_{-1}; s_{-1,i}) = - I^\#(X_{-1},\nu_{-1};s_{-1,-1-i}). \]
But then $I^\#(X_{-1},\nu_{-1})$ is equal to
\[ \sum_{i\geq 0} \left(I^\#(X_{-1},\nu_{-1}; s_{-1,i}) + I^\#(X_{-1},\nu_{-1};s_{-1,-1-i})\right) = 0, \]
so \eqref{eq:-1-triangle-untwisted} leads to $\dim I^\#(Y_0(K)) = \dim I^\#(Y_{-1}(K)) + \dim I^\#(Y) = d+1$.
\end{proof}

\section{Bundles over rational homology spheres} \label{sec:bundles}

Given a fixed 3-manifold $Y$, the invariant $I^\#(Y,\lambda)$ only depends on the homology class
\[ [\lambda] \in H_1(Y;\Z/2\Z), \]
when viewed as a $\Z/2\Z$-graded module up to isomorphism.  Theorem~\ref{thm:nu-zero-surgery} says that this dependence is unavoidable when $Y$ is zero-surgery on a knot $K$ with $\cinvt(K)=0$, but otherwise we can ask how common such examples are.  In this section we prove the following.

\begin{theorem} \label{thm:surgery-lambda}
Given a knot $K \subset S^3$ and a nonzero rational number $p/q$, the dimension
\[ \dim I^\#(S^3_{p/q}(K),\lambda) \]
is independent of $\lambda$.
\end{theorem}

\begin{proof}
We take $p$ and $q$ to be relatively prime, with $q \geq 1$.  If $p$ is odd then
\[ H_1(S^3_{p/q}(K);\Z/2\Z) = 0, \]
so there is only one possible $[\lambda]$ and hence there is nothing to prove.  We can thus take $p$ to be even from now on, so that \[H_1(S^3_{p/q}(K);\Z/2\Z) \cong \Z/2\Z,\] generated by the image in $S^3_{p/q}(K)$ of a meridian of $K$.

In the case $q > 1$, we essentially repeat the proof of \cite[Proposition~4.5]{bs-concordance} verbatim.  The argument comes down to the splitting of an exact triangle
\[ \dots \to I^\#(S^3_{a/b}(K)) \to I^\#(S^3_{p/q}(K)) \to I^\#(S^3_{c/d}(K)) \xrightarrow{F} \dots, \]
where $p=a+c$ and $q=b+d$ and the slopes all have pairwise distance one.  The last claim means that
\[ |aq-pb| = |pd-cq| = |cb-ad| = 1, \]
and since $p$ is even it follows that $a$ and $c$ must be odd, so that $I^\#(S^3_{a/b}(K),\lambda)$ and $I^\#(S^3_{c/d}(K),\lambda)$ do not depend on $\lambda$.  We can thus choose $\lambda$ carefully in \eqref{eq:surgery-triangle-general} to replace the above triangle with
\[ \dots \to I^\#(S^3_{a/b}(K)) \to I^\#(S^3_{p/q}(K),\mu) \to I^\#(S^3_{c/d}(K)) \xrightarrow{F'} \dots, \]
and the proof that $F=0$ in \cite{bs-concordance} depends only on the smooth topology of the surgery cobordisms rather than on their respective bundles, so it follows that $F'=0$ as well.  Thus the splittings of these two triangles imply that
\[ \dim I^\#(S^3_{p/q}(K),\mu) = \dim I^\#(S^3_{p/q}(K)) \]
whenever $q > 1$.

In the remaining case, we have $q=1$, so that $p/q$ is a nonzero integer.  Taking $n=p-1$ in \eqref{eq:surgery-triangle-general} gives us an exact triangle
\[ \dots \to I^\#(S^3,\lambda) \to I^\#(S^3_{p-1}(K),\lambda \cup \mu) \to I^\#(S^3_p(K),\lambda) \to \dots, \]
where $\mu$ is the image of a meridian of $K$.  The first two terms are independent of $\lambda$ because $H_1(S^3;\Z/2\Z) = 0$ and $H_1(S^3_{p-1}(K);\Z/2\Z) = 0$, so that
\begin{equation} \label{eq:dim-p-1}
\dim I^\#(S^3_p(K),\lambda) = \dim I^\#(S^3_{p-1}(K)) \pm 1.
\end{equation}
By the same argument, taking $n=p$ and replacing $\lambda$ with $\lambda \cup \mu$ in \eqref{eq:surgery-triangle-general} gives us 
\begin{equation} \label{eq:dim-p+1}
\dim I^\#(S^3_p(K),\lambda) = \dim I^\#(S^3_{p+1}(K)) \pm 1.
\end{equation}
Now since $p$ is nonzero and even, Theorem~\ref{thm:nu-odd} tells us that $p \neq \cinvt(K)$, so then
\[ \dim I^\#(S^3_{p+1}(K)) = \dim I^\#(S^3_{p-1}(K)) + \begin{cases} \hphantom{-}2, & p > \cinvt(K) \\ -2, & p < \cinvt(K) \end{cases} \]
by Theorem~\ref{thm:dim-surgery}.  We combine this with \eqref{eq:dim-p-1} and \eqref{eq:dim-p+1} to conclude that
\[ \dim I^\#(S^3_p(K),\lambda) = \dim I^\#(S^3_{p-1}(K)) + \begin{cases} \hphantom{-}1, & p > \cinvt(K) \\ -1, & p < \cinvt(K) \end{cases} \]
regardless of the choice of $\lambda$.
\end{proof}

\section{The kernel of $\cinvt$ and an instanton Floer epsilon invariant} \label{sec:kernel-nu}

According to Corollary~\ref{cor:nu-0-twisted}, knots $K\subset S^3$ with $\cinvt(K)=0$ fall into exactly one of two classes:
\begin{itemize}
\item \emph{W-shaped} knots, which satisfy
\[ \dim I^\#(S^3_0(K)) = r_0(K)+2 \quad\text{ and }\quad \dim I^\#(S^3_0(K),\mu) =r_0(K); \]
\item \emph{V-shaped} knots, which satisfy
\[ \dim I^\#(S^3_0(K)) = r_0(K) \quad\text{ and }\quad \dim I^\#(S^3_0(K),\mu) =r_0(K)+2. \]
\end{itemize}
The names come from the shape of the graph of $\dim I^\#(S^3_n(K))$ as $n$ varies among the integers.  These can be distinguished by the maps
\begin{align*}
I^\#(X_0,\nu_0)&: I^\#(S^3) \to I^\#(S^3_0(K)) \\
I^\#(X_0,\tilde\nu_0)&: I^\#(S^3) \to I^\#(S^3_0(K),\mu),
\end{align*}
where $X_0$ is the trace of $0$-surgery on $K$; if $\Sigma_0 \subset X_0$ is a capped-off Seifert surface generating $H_2(X_0)$ then $\nu_0\cdot\Sigma_0$ and $\tilde\nu_0\cdot\Sigma_0$ are even and odd, respectively.  (We will occasionally add ``$K$'' to the notation for clarity, as in $X_0(K)$, $\Sigma_{K,0}$, $\nu_{K,0}$, and $\tilde\nu_{K,0}$.)  According to the exact triangles \eqref{eq:d+1-1} and \eqref{eq:d+1-3}, these maps are respectively injective and zero if $K$ is W-shaped, while they are zero and injective if $K$ is V-shaped.

\begin{lemma} \label{lem:zero-surgery-sum}
Let $K$ and $L$ be knots in $S^3$.  Fix $\nu_K \in \{\nu_{K,0},\tilde\nu_{K,0}\}$ and $\nu_L \in \{\nu_{L,0},\tilde\nu_{L,0}\}$.  If $I^\#(X_0(K),\nu_K)$ and $I^\#(X_0(L),\nu_L)$ are both injective, then so is $I^\#(X_0(K\#L),\nu_{K\#L})$, where $\nu_{K\#L} \in \{\nu_{K\#L,0},\tilde\nu_{K\#L,0}\}$ satisfies
\[ \nu_{K\#L} \cdot \Sigma_{K\#L,0} \equiv \nu_K\cdot\Sigma_{K,0} + \nu_L\cdot\Sigma_{L,0} \pmod{2}. \]
\end{lemma}

\begin{proof}
This is a special case of \cite[Lemma~5.2]{bs-concordance}, which is proved by embedding $X_0(K\#L)$ into $X_0(K) \natural X_0(L)$ so that the cobordism map \[I^\#(X_0(K),\nu_{K,0}) \otimes I^\#(X_0(L),\nu_{L,0})\] is equal to a composition
\begin{align*}
I^\#(S^3) \xrightarrow{I^\#(X_0(K\#L),\nu_{K\#L,0})} I^\#(S^3_0(K\#L)) &\xrightarrow{\phantom{\cong}} I^\#(S^3_0(K)\#S^3_0(L)) \\
& \xrightarrow{\cong} I^\#(S^3_0(K)) \otimes I^\#(S^3_0(L)),
\end{align*}
and likewise if we replace either of $\nu_{K,0}$ and $\nu_{L,0}$ with $\tilde\nu_{K,0}$ and $\tilde\nu_{L,0}$.
\end{proof}

\begin{proposition} \label{prop:ker-nu-group}
Let $K$ and $L$ be knots in $S^3$ with $\cinvt(K)=\cinvt(L)=0$.  Then $\cinvt(K\#L) = 0$.  Moreover, $K\#L$ is W-shaped if $K$ and $L$ are both W-shaped or both V-shaped, and it is $V$-shaped otherwise.
\end{proposition}

\begin{proof}
Choose $\nu_K \in \{\nu_{K,0},\tilde\nu_{K,0}\}$ and $\nu_L \in \{\nu_{L,0},\tilde\nu_{L,0}\}$ so that the maps
\[ I^\#(X_0(K),\nu_K) \quad\text{and}\quad I^\#(X_0(L),\nu_L) \]
are both injective rather than zero.  Then Lemma~\ref{lem:zero-surgery-sum} says that $I^\#(X_0(K\#L),\nu_{K\#L})$ is injective as well, where
\[ \nu_{K\#L} = \begin{cases} \nu_{K\#L,0}, & K\text{ and }L\text{ have the same shape} \\ \tilde\nu_{K\#L,0}, & \text{otherwise}. \end{cases} \]

If $\nu_{K\#L} = \nu_{K\#L,0}$, then the exact triangle \eqref{eq:d+1-1} now says that
\[ \dim I^\#(S^3_0(K\#L)) = \dim I^\#(S^3_1(K\#L)) + 1. \]
This implies that either $\cinvt(K\#L)=0$ and $K\#L$ is W-shaped, or $\cinvt(K\#L) > 0$.  We apply the same argument to the mirrors $\mirror{K}$ and $\mirror{L}$, which have the same shapes as $K$ and $L$, to see that $-\cinvt(K\#L) = \cinvt(\mirror{K}\#\mirror{L}) \geq 0$.  So $\cinvt(K\#L) = 0$ and $K\#L$ is W-shaped after all.

Similarly, if $\nu_{K\#L} = \tilde\nu_{K\#L,0}$, then the exact triangle \eqref{eq:d+1-3} says that
\[ \dim I^\#(S^3_0(K\#L),\mu) = \dim I^\#(S^3_1(K\#L)) + 1, \]
hence by Corollary~\ref{cor:nu-0-twisted}, we have
\[ \dim I^\#(S^3_0(K\#L)) = \dim I^\#(S^3_1(K\#L)) - 1. \]
Thus either $\cinvt(K\#L) = 0$ and $K\#L$ is V-shaped, or $\cinvt(K\#L) < 0$.  Again we repeat this argument with $\mirror{K}$ and $\mirror{L}$ to conclude that $\cinvt(K\#L) = 0$ and $K\#L$ is V-shaped.
\end{proof}

Proposition~\ref{prop:ker-nu-group} shows that the subset
\[ \ker \cinvt = \{ [K] \in \cC \mid \cinvt(K)=0 \} \]
of the smooth concordance group is in fact a subgroup of $\cC$, and that the shape of a knot in $\ker\cinvt$ defines a homomorphism
\begin{align*}
\text{shape}: \ker \cinvt &\to \Z/2\Z \\
K &\mapsto \begin{cases} 0, & K\text{ is W-shaped} \\ 1, & K\text{ is V-shaped}. \end{cases}
\end{align*}
Thus the subgroup $\cC_W := \ker(\text{shape}) \subset \ker \cinvt$, consisting of smooth concordance classes of W-shaped knots, is either an index-2 subgroup of $\ker \cinvt$ or all of $\ker\cinvt$, depending on whether or not V-shaped knots with $\cinvt=0$ exist.  Combining this with Lemma~\ref{lem:tau-nonzero}, which implies that if $\chominvt(K) \neq 0$ then $\cinvt(K) \neq 0$, we have an ascending chain of subgroups
\[ \cC_W \subset \ker \cinvt \subset \ker \chominvt \subset \cC. \]
We do not know if either of the first two inclusions are proper.

\begin{remark}
It follows from work of Hom \cite{hom-epsilon} that the analogous inclusion
\[ \ker \hat\nu \subset \ker \tau \]
in Heegaard Floer homology is in fact proper.  Let $K$ be any knot of genus $g\geq 1$ satisfying $\tau(K) = g$, such as a positive torus knot, and let $K'$ denote its $(1-4g,2)$-cable (where $2$ is the longitudinal winding).  Then by \cite[Proposition~3.6(4)]{hom-epsilon} the epsilon invariant of $K$ satisfies $\epsilon(K) = 1$, so we can apply \cite[Theorems 1~and~2]{hom-epsilon} to see that
\[ \tau(K') = 2\tau(K) + \tfrac{1}{2}(2-1)((1-4g)-1) = 0 \]
and $\epsilon(K')=1$.  This guarantees that $K' \in \ker\tau$ but $K' \not\in \ker\hat\nu$ as follows: according to \cite[Remark~3.5]{hom-epsilon} we have
\[ \epsilon(K') = \big(\tau(K')-\nu(K')\big) - \big(\tau(\mirror{K'})-\nu(\mirror{K'})\big), \]
where $\nu$ is the invariant of \cite[Definition~9.1]{osz-rational}, so $\nu(\mirror{K'})- \nu(K') = 1$.  But from \cite[Equation~(34)]{osz-rational} we have $\nu(K') = \tau(K') \text{ or } \tau(K')+1$, and likewise for $\mirror{K'}$, so since $\tau(K')=\tau(\mirror{K'})=0$ the only possibility is
\begin{align*}
\nu(K') &= 0, &
\nu(\mirror{K'}) &= 1.
\end{align*}
Thus \cite[Lemma~10.4]{bs-concordance} says that $\hat\nu(K')=-1$, so $K' \not\in \ker\hat\nu$ as claimed.
\end{remark}

Finally, we show that taking connected sums with knots in $\ker\cinvt$ does not change the value of $\cinvt$.

\begin{proposition} \label{prop:add-nu-zero}
Let $K$ and $L$ be knots in $S^3$.  If $\cinvt(K)=0$, then $\cinvt(K\#L) = \cinvt(L)$.
\end{proposition}

\begin{proof}
The case $\cinvt(L)=0$ is Proposition~\ref{prop:ker-nu-group}, so we may assume that $\cinvt(L) \neq 0$, and in particular that $\cinvt(L)$ is odd by Theorem~\ref{thm:nu-odd}.

The relation \eqref{eq:quasi-morphism}, asserting that $\cinvt$ is a quasi-morphism, simplifies here to
\[ \left| \cinvt(K\#L) - \cinvt(L) \right| \leq 1. \]
If $\cinvt(K\#L) \neq \cinvt(L)$ then it follows that $\cinvt(K\#L) = \cinvt(L)\pm1$ is even, so by Theorem~\ref{thm:nu-odd} it must be zero.  In this case we have
\[ \cinvt(K\#L) = \cinvt(K) = 0 \]
and $\cinvt(L) = \mp 1$.  Then $\cinvt(\mirror{K}) = -\cinvt(K) = 0$, so Proposition~\ref{prop:ker-nu-group} says that
\[ \cinvt((K\#L)\#\mirror{K}) = 0. \]
But $K\#L\#\mirror{K}$ is smoothly concordant to $L$, so $\cinvt(L) = 0$ as well and this is a contradiction.  It must therefore be true that $\cinvt(K\#L) = \cinvt(L)$ after all.
\end{proof}

Inspired by Hom's epsilon invariant \cite{hom-epsilon}, we can now define another concordance invariant by the formula
\[ \cepsilon(K) = 2\chominvt(K) - \cinvt(K). \]
It is clearly a concordance invariant, because $\chominvt$ and $\cinvt$ are, and it satisfies many of the same properties of $\epsilon$ which are listed after \cite[Corollary~3]{hom-epsilon}.

\begin{proposition} \label{prop:epsilon}
We have $\cepsilon(K) \in \{-1,0,1\}$ for all $K\subset S^3$.  It satisfies the following properties:
\begin{enumerate}
\item $\cepsilon(K) = 0$ if and only if $\cinvt(K)=0$. \label{i:ce-1}
\item If $\cepsilon(K)=0$ then $\chominvt(K) = 0$. \label{i:ce-2}
\item If $K$ is slice, then $\cepsilon(K) = 0$. \label{i:ce-3}
\item $\cepsilon(\mirror{K}) = -\cepsilon(K)$. \label{i:ce-4}
\item If $|\chominvt(K)| = g_s(K) > 0$, then $\cepsilon(K) = \mathrm{sgn}(\chominvt(K))$. \label{i:ce-5}
\item If $\cepsilon(K)=0$, then $\cepsilon(K\#K') = \cepsilon(K')$. \label{i:ce-6}
\item If $\cepsilon(K)=\cepsilon(K')$, then $\cepsilon(K\#K') = \cepsilon(K)=\cepsilon(K')$. \label{i:ce-7}
\end{enumerate}
\end{proposition}

\begin{proof}
By \cite[Proposition~5.4]{bs-concordance}, we have an inequality
\[ \left| 2\chominvt(K) - \cinvt(K) \right| \leq 1 \]
which implies that $|\cepsilon(K)| \leq 1$, and Ghosh--Li--Wang \cite{glw} proved that it is an integer, so it must be $-1$, $0$, or $1$.  Then $\cepsilon(K)=0$ if and only if $\cinvt(K)$ is even, which by Theorem~\ref{thm:nu-odd} is equivalent to $\cinvt(K)=0$, establishing \eqref{i:ce-1}.  This in turn implies \eqref{i:ce-2} by Lemma~\ref{lem:tau-nonzero}.  Properties \eqref{i:ce-3} and \eqref{i:ce-4} follow from the same properties for $\chominvt$ and $\cinvt$.

For \eqref{i:ce-5}, since $g_s(K) > 0$ we have $|\cinvt(K)| \leq 2g_s(K)-1$ by \cite[Theorem~3.7]{bs-concordance}. If $\chominvt(K) = g_s(K)$ is positive then
\[ \cepsilon(K) = 2\chominvt(K) - \cinvt(K) \geq 2g_s(K) - (2g_s(K)-1) = 1, \]
so in fact $\cepsilon(K) = 1$, and the case $\chominvt(K) = -g_s(K)$ is similar. 

For \eqref{i:ce-6}, if $\cepsilon(K)=0$ then we know by \eqref{i:ce-1} and \eqref{i:ce-2} that $\cinvt(K)=\chominvt(K)=0$, so $\cinvt(K\#K')=\cinvt(K')$ and $\chominvt(K\#K')=\chominvt(K')$ by Proposition~\ref{prop:add-nu-zero} and the additivity of $\chominvt$ respectively, hence
\[ \cepsilon(K\#K') = 2\chominvt(K\#K') - \cinvt(K\#K') = 2\chominvt(K')-\cinvt(K') = \cepsilon(K'). \]

For \eqref{i:ce-7}, we may assume that $\cepsilon(K)=\cepsilon(K')=\pm1$ since otherwise this is a special case of \eqref{i:ce-6}.  By taking mirrors throughout and applying \eqref{i:ce-4} as needed, we may further assume that $\cepsilon(K)=\cepsilon(K)=1$, and we wish to show that $\cepsilon(K\#K')=1$ as well.  We take
\[ 2\chominvt(K\#K') = 2\chominvt(K) + 2\chominvt(K') \]
and rearrange this to get
\[ \cepsilon(K\#K') - \cepsilon(K) - \cepsilon(K') = -\left( \cinvt(K\#K') - \cinvt(K) - \cinvt(K') \right), \]
where the right side is $-1$, $0$, or $1$ by \eqref{eq:quasi-morphism}, and thus $\cepsilon(K\#K') - 2 \geq -1$, or $\cepsilon(K\#K') \geq 1$.  But this is only possible if $\cepsilon(K\#K')=1$, as desired.
\end{proof}

Noting that $\ker(\cepsilon) = \ker(\cinvt)$ by Proposition~\ref{prop:epsilon}, and the latter is a subgroup of $\cC$ by Proposition~\ref{prop:ker-nu-group}, we can define a total ordering on the quotient
\[ \cC / \ker(\cinvt) = \cC / \ker(\cepsilon) \]
by setting 
\[ \llbracket K \rrbracket \geq \llbracket K' \rrbracket \ \Longleftrightarrow\ \cepsilon(K\#\mirror{K'}) \geq 0. \]
(We use $\llbracket K \rrbracket$ to denote the equivalence class in $\cC/\ker(\cepsilon)$ of the smooth concordance class $[K] \in \cC$ of a knot $K$.)
This is well-defined and satisfies the axioms of a total ordering by the properties in Proposition~\ref{prop:epsilon}; we omit the proof and refer instead to \cite[Proposition~4.1]{hom-ordering}.

It is possible that $\ker(\cinvt) = \ker(\chominvt)$, and then $\chominvt: \cC \to \R$ has image $\Z$ by \cite{glw}, so a total ordering on $\cC/\ker(\chominvt) \cong \Z$ would not be very interesting.  But we expect that this is not the case, namely that there are knots $K$ with $\chominvt(K)=0$ but $\cinvt(K) = \pm1$, and then $\cC / \ker(\cinvt)$ may be a much more complicated group.

\section{Applications to homology cobordism} \label{sec:homology-cobordism}

Let $\Theta^3_\Z$ denote the group of integral homology 3-spheres modulo integral homology cobordism.  In recent work, Nozaki, Sato, and Taniguchi proved the following.

\begin{theorem}[{\cite[Theorem~1.8]{nozaki-sato-taniguchi}}] \label{thm:nst}
Let $K$ be a knot in $S^3$.  If $h(S^3_1(K)) < 0$, where $h$ denotes the Fr{\o}yshov invariant \cite[\S8]{froyshov}, then the homology spheres
\[ S^3_{1/n}(K), \quad n \geq 1 \]
are linearly independent in $\Theta^3_\Z$.
\end{theorem}

In this section we will relate $h(S^3_1(K))$ to the invariant $\cinvt(K)$ as follows.

\begin{proposition} \label{prop:h-1-surgery}
If $\cinvt(K) > 0$, then $h(S^3_1(K)) < 0$.
\end{proposition}

Postponing the proof of Proposition~\ref{prop:h-1-surgery} for now, we recall once again that $\chominvt$ is the homogenization
\[ \chominvt(K) = \frac{1}{2} \lim_{n\to\infty} \frac{\cinvt(\#^n K)}{n} \]
of $\cinvt(K)$, which defines a real-valued homomorphism on the smooth concordance group.  Then Proposition~\ref{prop:h-1-surgery} has the following corollary. 

\begin{theorem} \label{thm:homology-cobordism}
Let $K \subset S^3$ be a knot satisfying $\chominvt(K) > 0$, or more generally $\cinvt(K) > 0$.  Then the homology spheres
\[ S^3_{1/n}(K), \quad n \geq 1 \]
are linearly independent in $\Theta^3_\Z$.
\end{theorem}

\begin{proof}
Lemma~\ref{lem:tau-nonzero} asserts that if $\chominvt(K) > 0$ then $\cinvt(K) > 0$ as well.  In the latter case Proposition~\ref{prop:h-1-surgery} says that $h(S^3_1(K)) < 0$, so we apply Theorem~\ref{thm:nst}.
\end{proof}

The reason for the emphasis on $\chominvt(K)$ here is that as a concordance homomorphism, it is somewhat better behaved than $\cinvt(K)$, and in particular it is often easier to compute.  Nothing in this section will require Theorem~\ref{thm:main-nu-odd}, except for the implication
\[ \chominvt(K) > 0 \ \Longrightarrow\ \cinvt(K) > 0 \]
in Lemma~\ref{lem:tau-nonzero}.

We now prove Proposition~\ref{prop:h-1-surgery}.  In order to do so, we introduce several other versions of instanton homology:
\begin{itemize}
\item If $Y$ is a homology sphere, then $I(Y)$ is Floer's original instanton homology \cite{floer-instanton}, and $\hat{I}(Y)$ is Fr{\o}yshov's reduced instanton homology \cite{froyshov}.  Both of these are $\Z/8\Z$ graded, and they are mod 4 periodic with respect to this grading \cite[Corollary~3]{froyshov}.   We will follow the grading conventions of \cite{saveliev,scaduto}, which differ from those of \cite{floer-instanton}; then Fr{\o}yshov's $h$ invariant is defined in \cite[\S8]{froyshov} as
\[ h(Y) = -\frac{1}{2}\left(\chi(I(Y)) - \chi(\hat{I}(Y))\right). \]

\item If $Y$ is a homology $S^1\times S^2$ and $\mu$ generates $H_1(Y)$, then $I(Y)_\mu$ is the instanton homology associated to a principal $SO(3)$ bundle $P\to Y$ with $w_2(P) = \mathit{PD}(\mu)$, as in \cite{floer-surgery,braam-donaldson}.  It has a relative $\Z/8\Z$ grading which reduces to an absolute $\Z/2\Z$ grading.
\end{itemize}

The following is a special case of Fukaya's connected sum formula for instanton homology, as described by Scaduto \cite[Theorem~1.3]{scaduto}.  All coefficients are taken in $\Q$.

\begin{theorem}[\cite{scaduto}] \label{thm:fukaya-sum-general}
Let $Y$ be an integer homology sphere.  Then
\[ I^\#(Y) \cong H_*(\pt) \oplus \left(H_*(S^3) \otimes \ker(u^2-64)\right) \]
as absolutely $\Z/4\Z$-graded modules, where $u^2-64$ acts on $\bigoplus_{j=0}^3 \hat{I}_j(Y)$.

If instead $Y$ is a homology $S^1\times S^2$ and $\mu$ generates $H_1(Y;\Z)$, then
\[ I^\#(Y,\mu) \otimes H_*(S^4) \cong \ker(u^2-64) \otimes H_*(S^3) \]
as relatively $\Z/4\Z$-graded modules, where $u^2-64$ acts on $I(Y)_\mu$.
\end{theorem}

\begin{proof}[Proof of Proposition~\ref{prop:h-1-surgery}]
Floer's exact triangle \cite{floer-surgery,braam-donaldson} relates the instanton homologies of $S^3_0(K)$ and $S^3_1(K)$ as follows: there is an exact triangle between the unreduced homologies
\begin{equation} \label{eq:floer-triangle}
\dots \to I(S^3) \to I(S^3_0(K))_\mu \to I(S^3_1(K)) \to \dots,
\end{equation}
which remains exact at the homology spheres after we pass to reduced instanton homology, as explained in the proof of \cite[Theorem~10]{froyshov}.  Since $I(S^3) = \hat{I}(S^3) = 0$, this means that the resulting
\[ F: I(S^3_0(K))_\mu \twoheadrightarrow \hat{I}(S^3_1(K)) \]
is surjective, and its kernel has dimension
\begin{align*}
\dim_\Q(I(S^3_0(K))_\mu) - \dim_\Q \hat{I}(S^3_1(K)) &= \dim_\Q I(S^3_1(K)) - \dim_\Q \hat{I}(S^3_1(K)) \\
&= 2|h(S^3_1(K))|
\end{align*}
over $\Q$.  Here the first equality uses the fact that $I(S^3_0(K))_\mu \cong I(S^3_1(K))$ by the exact triangle \eqref{eq:floer-triangle}, while the second equality follows from the discussion in \cite[\S8]{froyshov}.

Since $-S^3_1(K) \cong S^3_{-1}(\mirror{K})$ bounds the trace of $(-1)$-surgery on $\mirror{K}$, which has negative definite intersection form, it follows from \cite[Theorem~3]{froyshov} that $h(-S^3_1(K)) \geq 0$, so $h(S^3_1(K)) \leq 0$ and thus $\dim_\Q \ker(F) = -2h(S^3_1(K))$.  If $h(S^3_1(K)) = 0$, then $F$ is also injective, so it's an isomorphism of $\Q[u]$-modules.  It thus restricts to an isomorphism
\[ \ker\left(u^2-64: \vphantom{\hat{I}}I(S^3_0(K))_\mu \to I(S^3_0(K))_\mu\right) \cong \ker\left(u^2-64: \hat{I}(S^3_1(K)) \to \hat{I}(S^3_1(K))\right), \]
so we conclude from Theorem~\ref{thm:fukaya-sum-general} that
\[ \dim I^\#(S^3_0(K),\mu) + 1 = \dim I^\#(S^3_1(K)). \]
But if $\cinvt(K) > 0$ then this contradicts Theorem~\ref{thm:dim-surgery}, so $\cinvt(K) > 0$ must imply that $h(S^3_1(K)) < 0$.
\end{proof}

\begin{remark} \label{rem:V-shaped-h}
If $\cinvt(K)=0$ and $K$ is V-shaped then (appealing to Corollary~\ref{cor:nu-0-twisted}) we get the same contradiction as in the proof of Proposition~\ref{prop:h-1-surgery}, so that $h(S^3_1(K)) < 0$ for such knots as well.  But then the same argument applies to the mirror $\mirror{K}$, so $-h(S^3_{-1}(K)) = h(S^3_1(\mirror{K})) < 0$.  This shows that if $\cinvt(K)=0$ and $K$ is V-shaped, then both $h(S^3_{-1}(K)) > 0$ and $h(S^3_1(K)) < 0$ must hold.
\end{remark}

We conclude this section with another application of Proposition~\ref{prop:h-1-surgery}.  Recall that a knot $K \subset S^3$ is \emph{rationally slice} if it bounds a smoothly embedded disk in some rational homology ball.

\begin{proposition} \label{prop:rationally-slice}
If $K \subset S^3$ is rationally slice, then $\cinvt(K) = \chominvt(K) = 0$ and $K$ is W-shaped.
\end{proposition}

\begin{proof}
If $K$ is rationally slice then $S^3_1(K)$ bounds some smooth rational homology ball $X$, and since $X$ has negative definite (in fact, trivial) intersection form, the Fr{\o}yshov invariant of its boundary satisfies $h(S^3_1(K)) \geq 0$ by \cite[Theorem~3]{froyshov}.  Proposition~\ref{prop:h-1-surgery} then says that we must have $\cinvt(K) \leq 0$.  But the mirror $\mirror{K}$ is also rationally slice, so the same argument says that $-
\cinvt(K) = \cinvt(\mirror{K}) \geq 0$ and we conclude that $\cinvt(K) = 0$.  Lemma~\ref{lem:tau-nonzero} implies that $\chominvt(K)=0$ as well.

The assertion that $K$ is W-shaped now follows from Remark~\ref{rem:V-shaped-h}: if it were V-shaped then we would have $h(S^3_1(K)) < 0$, but we have just argued that $h(S^3_1(K)) \geq 0$ instead.
\end{proof}

\section{Knots with small $r_0(K)$} \label{sec:almost-knots}

As mentioned in the introduction, the invariants $\cinvt(K)$ and $r_0(K)$ satisfy \[r_0(K) \geq |\cinvt(K)|\textrm{ and }r_0(K) \equiv \cinvt(K) \pmod{2},\] essentially by definition \cite[Definition~3.6]{bs-concordance}.    For nontrivial knots, we have $r_0(K) = \cinvt(K)$ if and only if $K$ is an \emph{instanton L-space knot}, meaning some positive rational surgery on $K$ is an instanton L-space, in which case \cite[Theorem~1.15]{bs-lspace} says that $K$ is fibered and strongly quasipositive and that
\[ r_0(K) = \cinvt(K) = 2g(K)-1. \]
In this section, we study knots $K\subset S^3$  with $r_0(K)$ small that just barely fail to be instanton L-space knots, by which we mean that \[r_0(K) - \cinvt(K) = 2.\] In doing so, we classify knots with $r_0(K)\leq 2$, proving Theorem \ref{thm:main-r0-small},  a key input in the bound on surgery slopes in Theorem \ref{thm:q-bound}, as explained in the introduction.

We continue to use the same notation as in Section~\ref{sec:zero-surgery-maps}, and in particular we define elements
\[ y_i \in I^\#_\godd(S^3_0(K),\mu;t_i) \]
exactly as in \eqref{eq:y_i}.  Throughout this section we will repeatedly use the following fact.

\begin{lemma}[{\cite[Theorem~1.17]{bs-lspace}}] \label{lem:fiber-detection}
If $K \subset S^3$ is a nontrivial knot of genus $g \geq 1$, then
\[ \dim I^\#_\godd(S^3_0(K),\mu;t_{g-1}) \geq 1 \]
with equality if and only if $K$ is fibered.
\end{lemma}

We begin with the following proposition:

\begin{proposition} \label{prop:02-13}
Let $K \subset S^3$ be a knot with $r_0(K)-\cinvt(K)=2$ and $\cinvt(K) \leq 1$.  Then $K$ has Seifert genus 1.  Moreover, $\cinvt(K) \geq -1$ with equality if and only if $K$ is the left-handed trefoil.
\end{proposition}

\begin{proof}
We note that since $r_0(K) > 0$ for all nontrivial $K$, we must have
\[ \cinvt(K) = r_0(K)-2 \geq -1. \] Furthermore, we have that \[ (\cinvt(K), r_0(K)) = (-1,1) \text{ or } (0,2) \text{ or } (1,3). \]
If $(\cinvt(K),r_0(K)) = (-1,1)$ then $\mirror{K}$ is a genus-1 instanton L-space knot, so $K$ must be the left-handed trefoil \cite[Corollary~7.13]{bs-lspace}.  We may therefore assume from now on that $(\cinvt(K),r_0(K))$ is either $(0,2)$ or $(1,3)$.

Let $g$ be the genus of $K$, and suppose that $g \geq 2$.  Then the homomorphisms $t_{1-g}$, $t_0$, and $t_{1-g}$ are distinct. Moreover, the $t_{g-1}$- and $t_{1-g}$-eigenspaces of $I^\#(S^3_0(K),\mu)$ are isomorphic as $\Z/2\Z$-graded modules by any of the conjugation symmetries of \eqref{eq:conjugation-y}, so by Lemma~\ref{lem:fiber-detection} we must have
\[ \dim I^\#_\godd(S^3_0(K),\mu;t_{1-g}) + \dim I^\#_\godd(S^3_0(K),\mu;t_{g-1}) \geq 2. \]
We use \eqref{eq:chi} to compute that
\[ \dim I^\#(S^3_0(K),\mu) = 2\dim I^\#_\godd(S^3_0(K),\mu) \geq 4, \]
and if equality holds then $I^\#_\godd(S^3_0(K),\mu;t_0) = 0$, hence the element $y_0$ is zero.

Now Theorem~\ref{thm:dim-surgery} tells us that $\dim I^\#(S^3_1(K)) = 3$, so by the exactness of \eqref{eq:surgery-triangle-0}, we must have
\[ \dim I^\#(S^3_0(K),\mu) \leq 4, \]
hence in fact equality holds and the map
\[ I^\#(X_0,\tilde\nu_0): I^\#(S^3) \to I^\#(S^3_0(K),\mu) \]
must be injective.  But since $\cinvt(K) \geq 0$ and $I^\#(X_0,\tilde\nu_0)$ is injective, Proposition~\ref{prop:y0-nonzero} tells us that $y_0$ is nonzero after all, and we have a contradiction.
\end{proof}

We would like to understand knots with $(\cinvt(K),r_0(K))$ equal to $(0,2)$ or $(1,3)$.  The only known examples are the figure eight and $\mirror{5_2}$, per \cite[Table~1]{bs-concordance}, and we expect that there are no others.  In order to find further restrictions on such knots, we will study their instanton homology in more detail.  In what follows we use the same notation as in Section~\ref{sec:homology-cobordism} for other versions of instanton homology, notably $I(Y)$ and $\hat{I}(Y)$ for an integral homology sphere and $I(Y)_\mu$ for a homology $S^1\times S^2$.

The connected sum formulas of Theorem~\ref{thm:fukaya-sum-general} simplify further for certain surgeries on knots, as described in \cite[Corollary~1.5]{scaduto} for $\pm1$-surgeries and \cite[\S9.8]{scaduto} for zero-surgeries.

\begin{theorem}[\cite{scaduto}] \label{thm:fukaya-sum}
Let $K \subset S^3$ be a knot of Seifert genus at most 2.  Then
\[ I^\#(S^3_{\pm1}(K)) \cong H_*(\pt) \oplus \left(H_*(S^3) \otimes \bigoplus_{j=0}^3 \hat{I}_j(S^3_{\pm1}(K))\right) \]
as absolutely $\Z/4\Z$-graded modules, and
\[ I^\#(S^3_0(K),\mu) \otimes H_*(S^4) \cong I(S^3_0(K))_{\mu} \otimes H_*(S^3) \]
as relatively $\Z/4\Z$-graded modules.
\end{theorem}

Cobordism maps of the form $I^\#(X,\nu): I^\#(Y_1) \to I^\#(Y_2)$ are homogeneous with respect to the $\Z/4\Z$ grading: according to \cite[Proposition~7.1]{scaduto}, we have
\begin{equation} \label{eq:cobordism-degree}
\deg I^\#(X,\nu) \equiv -\frac{3}{2}(\chi(X)+\sigma(X)) + \frac{1}{2}(b_1(Y_2)-b_1(Y_1)) + 2[\nu]^2 \pmod{4}.
\end{equation}
This will be useful in proving the following proposition.

\begin{proposition} \label{prop:almost-cinvt-0}
Suppose that $(\cinvt(K),r_0(K)) = (0,2)$.  Then $K$ is the figure eight, and it satisfies
\begin{align*}
\dim I^\#(S^3_0(K),\mu) &= 2, & \dim I^\#(S^3_0(K)) &= 4.
\end{align*}
\end{proposition}

\begin{proof}
We note from Proposition~\ref{prop:02-13} that $K$ has Seifert genus $1$, and that
\[ \{ \dim I^\#(S^3_0(K)), \dim I^\#(S^3_0(K),\mu) \} = \{ 2, 4 \} \]
by Corollary~\ref{cor:nu-0-twisted}.  If $\dim I^\#(S^3_0(K),\mu) = 2$ (which means that $\dim I^\#(S^3_0(K))=4$), then by \eqref{eq:chi} we know that $I^\#_\godd(S^3_0(K),\mu)$ is 1-dimensional, which by Lemma~\ref{lem:fiber-detection} is equivalent to $K$ being fibered.  If $K$ were a trefoil then we would have $(\cinvt(K),r_0(K)) = (\pm1,1)$, so $K$ must be the figure eight instead.

Assuming from now on that $K$ is not the figure eight, we must have $\dim I^\#(S^3_0(K),\mu)=4$ and so
\[ \dim I^\#(S^3_0(K)) = 2. \]
We will use this computation of $I^\#(S^3_0(K))$ to completely determine the instanton homology $I(S^3_{-1}(K))$, and hence the Alexander polynomial $\Delta_K(t)$.  We will then use this to get a lower bound on the dimension of $I(S^3_{-1}(C))$, where $C$ is the $(-1,2)$-cable of $K$, and show that the given values of $(\cinvt(K),r_0(K))$ prevent $I(S^3_{-1}(C))$ from being as large as required, leading to a contradiction.

To begin, Theorem~\ref{thm:fukaya-sum} implies that $I^\#(S^3_{\pm1}(K))$ always has a $\Q_0$ summand, where we use subscripts to denote the $\Z/4\Z$ grading.  The total dimension is 3 and the Euler characteristic is 1 by \eqref{eq:chi}, so there are two more $\Q$ summands, one in an odd grading and one in an even grading.  Any odd element of $\Z/4\Z$ is adjacent to any even element, so we can write
\begin{align*}
I^\#(S^3_{-1}(K)) &\cong \Q_0 \oplus \Q_{k-1} \oplus \Q_k \\
I^\#(S^3_{1}(K)) &\cong \Q_0 \oplus \Q_{m-1} \oplus \Q_m
\end{align*}
for some $k,m \in \Z/4\Z$.

We now examine the surgery exact triangles coming from the $n=-1$ and $n=0$ cases of \eqref{eq:surgery-triangle}, namely
\[ \dots \to I^\#(S^3) \xrightarrow{F_{-1}} I^\#(S^3_{-1}(K)) \xrightarrow{G_0} I^\#(S^3_0(K)) \xrightarrow{H_0} \dots \]
and
\[ \dots \to I^\#(S^3) \xrightarrow{F_0} I^\#(S^3_0(K)) \xrightarrow{G_1} I^\#(S^3_1(K)) \xrightarrow{H_1} \dots. \]
Each of these maps is induced by a 2-handle cobordism, and hence has a degree mod 4 given by \eqref{eq:cobordism-degree}.
These degrees are computed in \cite[Corollary~5.2]{lpcs}: we have
\begin{align*}
\deg(F_{-1}) &\equiv 2, & \deg(G_0) &\equiv 3, & \deg(H_0) &\equiv 2, \\
\deg(F_0) &\equiv 3, & \deg(G_1) &\equiv 2, & \deg(H_1) &\equiv 2.
\end{align*}
Since $I^\#(S^3) \cong \Q_0$ and $\dim I^\#(S^3_0(K)) = 2$ while $\dim I^\#(S^3_{\pm1}(K)) = 3$, the maps $H_0$ and $F_0$ are both zero.  In the first triangle, the image of $F_{-1}$ is a $\Q_2$ summand of $I^\#(S^3_{-1})$, so $k$ must be either $2$ or $3$ and then
\[ I^\#(S^3_{-1}(K)) \cong \begin{cases} \Q_0 \oplus \Q_1 \oplus \Q_2, & k=2 \\ \Q_0 \oplus \Q_2 \oplus \Q_3, & k=3 \end{cases}
\quad\Longrightarrow\quad
I^\#(S^3_0(K)) \cong \begin{cases} \Q_0 \oplus \Q_3, & k=2 \\ \Q_2 \oplus \Q_3, & k=3. \end{cases} \]
In the second triangle, the surjection $H_1$ must send a $\Q_2$ summand of $I^\#(S^3_1(K))$ onto $I^\#(S^3) \cong \Q_0$, so $m$ is either $2$ or $3$.  Then
\[ I^\#(S^3_{1}(K)) \cong \begin{cases} \Q_0 \oplus \Q_1 \oplus \Q_2, & m=2 \\ \Q_0 \oplus \Q_2 \oplus \Q_3, & m=3 \end{cases}
\quad\Longrightarrow\quad
I^\#(S^3_0(K)) \cong \begin{cases}
\Q_2 \oplus \Q_3, & m=2 \\
\Q_1 \oplus \Q_2, & m=3. \end{cases} \]
Since both computations must produce the same value of $I^\#(S^3_0(K))$, we have $k=3$ and $m=2$.

We have shown that $I^\#(S^3_{-1}(K)) \cong \Q_0 \oplus \Q_2 \oplus \Q_3$, so Theorem~\ref{thm:fukaya-sum} now tells us that
\[ \hat{I}(S^3_{-1}(K)) \cong \Q_3 \oplus \Q_7, \]
since $\hat{I}$ is mod 4 periodic.  Since $K$ has genus $1$, the Fr{\o}yshov $h$ invariant satisfies
\[ 0 \leq h(S^3_{-1}(K)) \leq 1 \]
by \cite[Lemma~9]{froyshov}.  If $h(S^3_{-1}(K))$ were zero then we would have $\hat{I}(S^3_{-1}(K)) \cong I(S^3_{-1}(K))$, but by Floer's exact triangle \cite{floer-surgery,braam-donaldson} the latter is isomorphic to $I(S^3_0(K))_\mu$, which has the same dimension as $I^\#(S^3_0(K),\mu) \cong \Q^4$ by Theorem~\ref{thm:fukaya-sum}.  Thus $h(S^3_{-1}(K)) = 1$, and by the definition of the $h$ invariant we conclude that
\[ \dim \hat{I}_j(S^3_{-1}(K)) = \dim I_j(S^3_{-1}(K)) - \begin{cases} 0, & j\neq 1,5 \\ 1, & j=1,5; \end{cases} \]
this follows implicitly from \cite[\S8]{froyshov}.  (It is claimed explicitly in \cite[\S9.3]{scaduto}, but for gradings $0,4$ instead of $1,5$ due to a typo; it is easy to check from the material immediately preceding that claim that the version claimed here is correct.)  Thus
\[ I(S^3_{-1}(K)) \cong \Q_1 \oplus \Q_3 \oplus \Q_5 \oplus \Q_7. \]
In particular we have
\[ -4 = \chi(I(S^3_{-1}(K))) = 2\lambda(S^3_{-1}(K)) = -\Delta''_K(1), \]
where $\lambda$ is the Casson invariant, so if $\Delta_K(t) = at + (1-2a) + at^{-1}$ then $a=2$.  In other words, $\Delta_K(t) = 2t-3+2t^{-1}$.

We now let $C = C_{-1,2}(K)$ be the $(-1,2)$-cable of $K$, represented by a curve in $\partial N(K)$ in the homotopy class $\mu^{-1}\lambda^2$.  Then $C$ has genus $2$ and Alexander polynomial
\[ \Delta_C(t) = \Delta_K(t^2) = 2t^2 - 3 + 2t^{-2}, \]
so that the instanton homology of $S^3_{-1}(C)$ has Euler characteristic
\[ \chi(I(S^3_{-1}(C))) = -\Delta''_C(1) = -16. \]
Applying Floer's exact triangle again, we have 
\begin{equation} \label{eq:0-surgery-C12}
\dim I(S^3_0(C))_\mu = \dim I(S^3_{-1}(C)) \geq 16.
\end{equation}
On the other hand, a result of Gordon \cite[Corollary~7.3]{gordon} says that
\[ S^3_{-1}(C) \cong S^3_{-1/4}(K), \]
and so we apply Theorem~\ref{thm:dim-surgery} with $(\cinvt(K),r_0(K))=(0,2)$ to get
\[ \dim I^\#(S^3_{-1}(C)) = \dim I^\#(S^3_{-1/4}(K)) = 9, \]
hence by applying the exact triangle \eqref{eq:surgery-triangle-0} to $\mirror{C}$ and reversing orientation we conclude that
\[ \dim I^\#(S^3_0(C),\mu) = 8 \text{ or } 10. \]
But now Theorem~\ref{thm:fukaya-sum} says that $\dim I(S^3_0(C))_\mu$ is 8 or 10, contradicting \eqref{eq:0-surgery-C12}.
\end{proof}

\begin{remark} \label{rem:3-1-alexander}
As discussed in Remark~\ref{rem:r0-equals-3}, the knots $K$ with $r_0(K)=3$ have also been classified since the initial appearance of this paper: the only such knots are $T_{2,5}$, $5_2$, and their mirrors.  This was proved by Farber, Reinoso, and Wang \cite{frw-cinquefoil} in the case $\cinvt(K) = \pm 3$, and by Li and Ye \cite[\S8]{li-ye-2} in the case $\cinvt(K)=\pm1$ following our work \cite[Theorem~3.13]{bs-characterizing-5_2} in the Heegaard Floer setting.  These results require substantially different techniques that are beyond the scope of this paper.
\end{remark}

As mentioned in the introduction, Proposition~\ref{prop:almost-cinvt-0} implies Theorem \ref{thm:main-r0-small} as a corollary:

\begin{proof}[Proof of Theorem \ref{thm:main-r0-small}]
Suppose $r_0(K) \leq 2$. Then we must have
\[ (\cinvt(K), r_0(K)) = (0,0) \text{ or } (\pm1,1) \text{ or } (0,2). \] As mentioned in the introduction, the first two cases correspond to the unknot and trefoils, respectively \cite[Proposition~7.12]{bs-lspace}. Proposition~\ref{prop:almost-cinvt-0} implies that $K$ is the figure eight knot in the third case.
\end{proof}

Proposition~\ref{prop:almost-cinvt-0} also allows us to understand when the framed instanton homology of zero-surgery is as small as possible.  In particular, it resolves a question from \cite{bs-concordance} about whether the figure eight knot is ``V-shaped'' or ``W-shaped'', by proving it to be W-shaped.

\begin{theorem} \label{thm:small-zero-surgery}
We have $\dim I^\#(S^3_0(K)) = 2$ if and only if $K$ is the unknot or a trefoil.
\end{theorem}

\begin{proof}
If $\cinvt(K) \neq 0$, then $\dim I^\#(S^3_0(K)) = 2$ if and only if $r_0(K) + |\cinvt(K)| = 2$, by Theorem~\ref{thm:dim-surgery}.  Since $r_0(K) \geq |\cinvt(K)| \geq 1$, this happens precisely when $r_0(K) = |\cinvt(K)| = \pm1$, so either $K$ or its mirror is a genus-one instanton L-space knot, and the only such knot is the right-handed trefoil \cite[Corollary~7.13]{bs-lspace}.

Assuming now that $\cinvt(K)=0$, we know by Corollary~\ref{cor:nu-0-twisted} that
\[ 2 = \dim I^\#(S^3_0(K)) \in \{ r_0(K), r_0(K)+2 \}, \]
so $r_0(K)$ is either $0$ or $2$.  If $r_0(K)=0$ then $K$ is the unknot, and indeed we have
\[ \dim I^\#(S^1\times S^2) = 2 \]
by \cite[\S7.6]{scaduto}.  But if $r_0(K)=2$ then Proposition~\ref{prop:almost-cinvt-0} says that $K$ would have to be the figure eight, for which $\dim I^\#(S^3_0(K))=4$ anyway.
\end{proof}

\bibliographystyle{alpha}
\bibliography{References}

\end{document}